\documentclass{amsart}

\usepackage[dvips]{graphicx}

\usepackage{charter}
\usepackage[T1]{fontenc}
\usepackage[a4paper,top=46mm,bottom=46mm,left=37mm,right=37mm]{geometry}

\usepackage{latexsym,amsfonts} 
\usepackage{amsmath,amssymb,graphics,setspace}
\usepackage{amsthm}
\usepackage{mathrsfs} 
\usepackage[T1]{fontenc}%

\usepackage[ruled,vlined,linesnumbered]{algorithm2e}

\usepackage{marvosym}

\usepackage{pstricks,pst-text,pst-grad,pst-node,pst-3dplot,pstricks-add,pst-poly,pst-coil,pst-bar,pst-all,pst-plot,pst-2dplot} 
\usepackage{pst-fun,pst-blur} 

\usepackage{filecontents}

\usepackage{picins,graphicx}
\usepackage{paralist}
\usepackage{color}

\usepackage[verbose]{wrapfig}

\usepackage{subfigure}
\renewcommand{\thesubfigure}{\thefigure.\arabic{subfigure}}
\makeatletter
\renewcommand{\p@subfigure}{}
\renewcommand{\@thesubfigure}{\thesubfigure:\hskip\subfiglabelskip}
\makeatother

\usepackage{hyperref}
\hypersetup{linktocpage=true,colorlinks=true,linkcolor=blue,citecolor=blue,pdfstartview={XYZ 1000 1000 1}}

\usepackage{floatflt,graphicx}

\usepackage{graphicx}
\makeatletter
\DeclareFontFamily{U}{tipa}{}
\DeclareFontShape{U}{tipa}{bx}{n}{<->tipabx10}{}
\newcommand{\arc@char}{{\usefont{U}{tipa}{bx}{n}\symbol{62}}}%

\newcommand{\arc}[1]{\mathpalette\arc@arc{#1}}

\newcommand{\arc@arc}[2]{%
  \sbox0{$\m@th#1#2$}%
  \vbox{
    \hbox{\resizebox{\wd0}{\height}{\arc@char}}
    \nointerlineskip
    \box0
  }%
}
\makeatother

\makeatletter
\newcommand{\doublewedge}{\big@doubleop{\wedge}}
\newcommand{\big@doubleop}[1]{%
  \DOTSB\mathop{\mathpalette\big@doubleop@aux{#1}}\slimits@
}

\newcommand\big@doubleop@aux[2]{%
  \sbox\z@{$\m@th#1#2$}%
  \makebox[1.35\wd\z@][s]{$\m@th#1#2\hss#2$}%
}

\newcommand{\norm}[1]{\left\|#1\right\|}  

\newcommand{\cl}{\mbox{cl}}
\newcommand{\dcl}{\mbox{cl}_{\Phi}}
\newcommand{\Int}{\mbox{int}}
\newcommand{\bdy}{\mbox{bdy}}

\newcommand{\Nrv}{\mbox{Nrv}}
\newcommand{\near}{\delta} 
\newcommand{\dnear}{\delta_{\Phi}} 
\newcommand{\dcap}{\mathop{\cap}\limits_{\Phi}} 
\newcommand{\dcup}{\mathop{\cup}\limits_{\Phi}} 

\newcommand{\skel}{\mbox{skel}}
\newcommand{\sk}{\mbox{sk}}
\newcommand{\sh}{\mbox{sh}}

\newcommand{\cyc}{\mbox{cyc}}
\newcommand{\vcyc}{\mbox{vcyc}}
\newcommand{\vNrv}{\mbox{vNrv}}

\newcommand{\sn}{\mathop{\delta}\limits^{\doublewedge}} 
\newcommand{\conn}{\mathop{\delta}\limits^{conn}} 
\newcommand{\farconn}{\mathop{\not{\delta}}\limits^{conn}} 
\newcommand{\sconn}{\mathop{\mathop{\delta}\limits^{conn}}\limits^{\doublewedge}} 
\newcommand{\dsconn}{\mathop{\mathop{\delta_{_{\Phi}}}\limits^{conn}}\limits^{\doublewedge}} 
\newcommand{\fardsconn}{\mathop{\mathop{\not{\delta}_{_{\Phi}}}\limits^{conn}}\limits^{\doublewedge}} 

\newcommand{\snd}{\mathop{\delta_{_{\Phi}}}\limits^{\doublewedge}} 
\newtheorem{example}{Example}

\newtheorem{definition}{Definition}
\newtheorem{lemma}{Lemma}
\newtheorem{theorem}{Theorem}

\newtheorem{corollary}{Corollary}

\newtheorem{problem}{Problem}

\usepackage{xcolor}
\definecolor{light}{gray}{0.80}

\begin{document}

\title[Proximal Vortex Cycles]{Proximal Vortex Cycles and Vortex Nerves.\\   Non-Concentric, Nesting, Possibly Overlapping\\ Homology Cell Complexes}

\author[James F. Peters]{James F. Peters}
\address{
Computational Intelligence Laboratory,
University of Manitoba, WPG, MB, R3T 5V6, Canada and
Department of Mathematics, Faculty of Arts and Sciences, Ad\.{i}yaman University, 02040 Ad\.{i}yaman, Turkey}
\thanks{The research has been supported by the Natural Sciences \&
Engineering Research Council of Canada (NSERC) discovery grant 185986 
and Instituto Nazionale di Alta Matematica (INdAM) Francesco Severi, Gruppo Nazionale per le Strutture Algebriche, Geometriche e Loro Applicazioni grant 9 920160 000362, n.prot U 2016/000036.}


\date{}

\dedicatory{Dedicated to  William Thomson and Som Naimpally}

\begin{abstract}
This article introduces proximal planar vortex 1-cycles, resembling the structure of vortex atoms introduced by William Thomson (Lord Kelvin) in 1867 and recent work on the proximity of sets that overlap either spatially or descriptively.      Vortex cycles resemble Thomson's model of a vortex atom, inspired by P.G. Tait's smoke rings.    A vortex cycle is a collection of non-concentric, nesting 1-cycles with nonempty interiors ({\em i.e.}, a collection of 1-cycles that share a nonempty set of interior points and which may or may not overlap).   Overlapping 1-cycles in a vortex yield an Edelsbrunner-Harer nerve within the vortex.   Overlapping vortex cycles constitute a vortex nerve complex.   Several main results are given in this paper, namely, a Whitehead CW topology and a Leader uniform topology are outcomes of having a collection of vortex cycles (or nerves) equipped with a connectedness proximity and the case where each cluster of closed, convex vortex cycles and the union of the vortex cycles in the cluster have the same homotopy type.
\end{abstract}

\maketitle
\tableofcontents

\section{Introduction}
This paper introduces vortex cycles restricted to the Euclidean plane.   Each vortex cycle $A$ (denoted by $\vcyc A$ (briefly, vortex $\vcyc A$)) is a collection of non-concentric, nesting 1-cycles with nonempty interiors ({\em i.e.}, 1-cycles that share a nonempty set of interior points and which may or may not overlap).   That is, the 1-cycles in every planar vortex cycle have a common nonempty interior.  A 1-cycle is a finite, collection of vertices (0-cells) connected by oriented edges (1-cells) that define a simple, closed path so that there is a path between any pair of vertices in the collection.   A path is simple, provided it has no self-intersections.   

Let $\vcyc A$ be a finite region of the Euclidean plane (denoted by $\mathbb{R}^2$).  Also, let $\bdy(\vcyc A)$ be a set of boundary points of $\vcyc A$.  Then, for every vortex cycle, there is a collection of functions $f:\bdy(\vcyc A)\longrightarrow \mathbb{R}^2$ such that each function maps a $\vcyc A$ boundary point to an interior fixed point shared by the 1-cycles in the vortex.   The physical analogue of a vortex cycle is a collection of non-concentric, nesting equipotential curves in an electric field~\cite[\S 5.1, pp. 96-97]{BaldomirHammond1996EMsystemGeometry}.   This view of vortex cycles befits a proximal physical geometry approach to the study of vortices in the physical world~\cite{Peters2016AMSJphysicalGeometry}.

Oriented 1-cycles by themselves in vortex cycles are closed braids~\cite{BirmanMenasco1992TAMSclosedBraids} with nonempty interiors.    The study of vortex cycles and their spatial as well as descriptive proximities is important in isolating distinctive shape properties such as vertex area, cycle overlap count, hole count, nerve count, perimeter, diameter over surface shape sub-regions.   A finite, bounded \emph{planar shape} $A$ (denoted by $\sh A$) is a finite region of the Euclidean plane bounded by a simple closed curve and with a nonempty interior~\cite{Peters2017arXiv1708-04147planarShapes}.   In effect, a vortex cycle is a system of shapes within a shape\footnote{Many thanks to M.Z. Ahmad for pointing this out.}

The geometry of vortex cycles is related to the study shape signatures~\cite{Peters2017AMSJshapeSignature} and the geometry of photon vortices by N.M. Litchinitser~\cite{Litchinitser2012structuredLightGeometry}, overlapping vortices by E. Adelberger, G. Dvali and A. Gruzinov~\cite{Adelberger2007nonzeroPhotonMass}, vortex properties of photons and electromagnetic vortices formed by photons by I.V. Dzedolik~\cite{Dzedolik2004vortexPropertiesOfPhotons} and vortex atoms introduced by Kelvin~\cite{Kelvin1867vortexAtoms}.

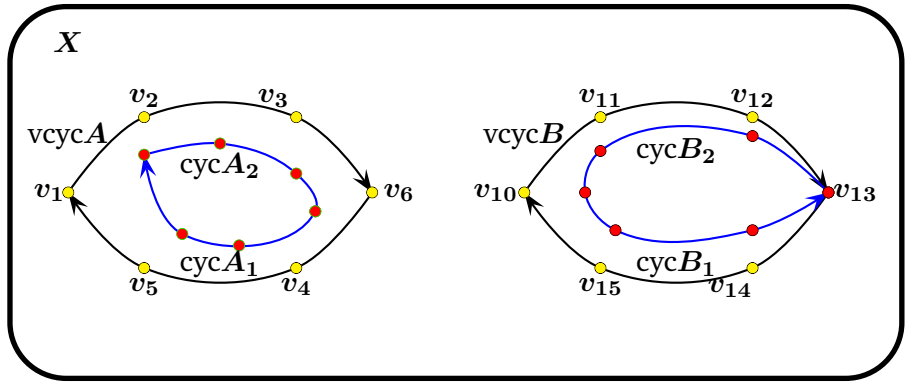
\begin{figure}[!ht]
\centering
\begin{pspicture}
(-0.5,-2.5)(11,3)
\psframe[linewidth=2pt,framearc=.3](-0.8,-2.5)(11.0,2.5)
\rput(0.0,2.0){\textcolor{black}{\large $\boldsymbol{X}$}}
\pscurve[linecolor=black,arrowscale=2]{->}(4,0)(3,-1)(1,-1)(0,0)
\pscurve[linecolor=black,arrowscale=2]{->}(0,0)(1,1)(3,1)(4,0)
\pscurve[linecolor=blue,arrowscale=2]{->}%
(1.00,0.5)(2.00,0.65)(3.0,0.25)(3.25,-0.25)(1.5,-0.55)(1.00,0.5)
\psdots[dotstyle=o, linewidth=1.2pt,linecolor = green, fillcolor = red]%
(3.25,-0.25)(3.0,0.25)(1.00,0.5)(2.00,0.65)(2.25,-0.70)(1.5,-0.55)(3.25,-0.25)
\rput(2.0,-0.95){\textcolor{black}{\large $\boldsymbol{\cyc A_1}$}}
\rput(2.0,0.35){\textcolor{black}{\large $\boldsymbol{\cyc A_2}$}}
\psdots[dotstyle=o, linewidth=1.2pt,linecolor = black, fillcolor = yellow]%
(0,0)(1,-1)(1,1)(3,1)(3,-1)(4,0)
\rput(2.7,1.25){\textcolor{black}{\large $\boldsymbol{v_3}$}}\rput(1,1.25){\textcolor{black}{\large $\boldsymbol{v_2}$}}
\rput(-0.25,0){\textcolor{black}{\large $\boldsymbol{v_1}$}}\rput(3,-1.25){\textcolor{black}{\large $\boldsymbol{v_4}$}}
\rput(0.0,0.75){\textcolor{black}{\large $\boldsymbol{\vcyc A}$}}
\rput(1,-1.25){\textcolor{black}{\large $\boldsymbol{v_5}$}}
\rput(4.35,0.0){\textcolor{black}{\large $\boldsymbol{v_6}$}}
\pscurve[linecolor=black,arrowscale=2]{->}(10,0)(9,-1)(7,-1)(6,0)
\pscurve[linecolor=black,arrowscale=2]{->}(6,0)(7,1)(9,1)(10,0)
\pscurve[linecolor=blue,arrowscale=2]{->}(10,0)(9.0,0.75)(7.0,0.55)(6.8,0.0)(7.2,-0.5)(9.0,-0.5)(10,0)
\psdots[dotstyle=o, linewidth=1.2pt,linecolor = black, fillcolor = yellow]%
(6,0)(7,1.0)(9,-1)(7,-1)(10,0)(9,1)
\psdots[dotstyle=o, linewidth=1.2pt,linecolor = black, fillcolor = red]%
(10,0)(9.0,0.75)(7.0,0.55)(6.8,0.0)(7.2,-0.5)(9.0,-0.5)(10,0)
\rput(5.62,0){\textcolor{black}{\large $\boldsymbol{v_{10}}$}}\rput(7,1.25){\textcolor{black}{\large $\boldsymbol{v_{11}}$}}
\rput(9,1.25){\textcolor{black}{\large $\boldsymbol{v_{12}}$}}\rput(10.35,0.0){\textcolor{black}{\large $\boldsymbol{v_{13}}$}}
\rput(8.7,-1.265){\textcolor{black}{\large $\boldsymbol{v_{14}}$}}\rput(7,-1.25){\textcolor{black}{\large $\boldsymbol{v_{15}}$}}
\rput(6,0.75){\textcolor{black}{\large $\boldsymbol{\vcyc B}$}}
\rput(8.0,-0.95){\textcolor{black}{\large $\boldsymbol{\cyc B_1}$}}
\rput(8.0,0.55){\textcolor{black}{\large $\boldsymbol{\cyc B_2}$}}
\end{pspicture}
\caption[]{Pair of  Two Different Vortex Cycles}
\label{fig:1-vortexCycles}
\end{figure}

Overlapping 1-cycles in a vortex constitute an Edelsbrunner-Harer nerve within the vortex.    Let $F$ be a finite collection of sets.   An \emph{Edelsbrunner-Harer nerve}~\cite[\S III.2, p. 59]{Edelsbrunner1999} nerve consists of all nonempty subcollections of $F$ (denoted by $\Nrv F$) whose sets have nonempty intersection, {\em i.e.},
\[
\Nrv F = \left\{X\subseteq F: \bigcap X\neq \emptyset\right\}\ \mbox{(Edelsbrunner-Harer Nerve)}.
\]

\begin{example}{\bf Two Forms of Vortex Cycles}.\\
Two different vortex cycles $vcyc A, \vcyc B$ are shown in Fig.~\ref{fig:1-vortexCycles}.    Vortex $vcyc A$ contains a pair of non-overlapping 1-cycles $\cyc A_1, \cyc A_2$.      By contrast, vortex $vcyc B$ in Fig.~\ref{fig:1-vortexCycles} contains a pair of overlapping 1-cycles $\cyc B_1, \cyc B _2$ with a common vertex, namely, $v_{13}$.   Let $F$ be a collection of sets of edges in $\cyc B_1, \cyc B _2$.   The pair of 1-cycles in vortex $\vcyc B$ constitute an Edelsbrunner-Harer nerve, since $\cyc B_1\cap \cyc B _2 = v_{13}$, {\em i.e.}, the intersection of 1-cycles $\cyc B_1, \cyc B _2$ is nonempty.   The edges of the cycles in both forms of vortices define closed convex curves.
\qquad \textcolor{blue}{\Squaresteel}
\end{example}

\noindent A number of simple results for vortex cycles come from the Jordan Curve Theorem.

\begin{theorem}\label{thm:JordonCurveTheorem} {\rm [Jordan Curve Theorem~\cite{Jordan1893coursAnalyse}]}.\\
A simple closed curve lying on the plane divides the plane into two regions and forms their common boundary.
\end{theorem}
\begin{proof}
For the first complete proof, see O. Veblen~\cite{Veblen1905TAMStheoryOFPlaneCurves}.  For a simplified proof via the Brouwer Fixed Point Theorem, see R. Maehara~\cite{Maehara1984AMMJordanCurvedTheoremViaBrouwerFixedPointTheorem}.  For an elaborate proof, see J.R. Mundres~\cite[\S 63, 390-391, Theorem 63.4]{Munkres2000}.
\end{proof}

\begin{lemma}\label{lemma:shape}
A finite planar shape contour separates the plane into two distinct regions.
\end{lemma}
\begin{proof}
The boundary of each planar shape is a finite, simple closed curve.   Hence, from Theorem~\ref{thm:JordonCurveTheorem},
a finite, planar shape separates the plane into two regions, namely, the region outside the shape boundary and the region in the shape interior.
\end{proof}

\begin{theorem}\label{thm:vortexCycle}
A finite planar vortex cycle is a collection of non-concentric, nesting shapes within a shape.
\end{theorem}
\begin{proof}
Each 1-cycle in a finite planar vortex cycle is a simple, closed curve.   By definition, a vortex cycle is a collection of  non-concentric 1-cycles nesting within a 1-cycle, each with a nonempty interior. From Theorem~\ref{thm:JordonCurveTheorem}, each vortex 1-cycle separates the plane into two regions.   Hence, from Lemma~\ref{lemma:shape}, a finite planar vortex is a collection of planar shapes within a shape.
\end{proof}

\begin{figure}[!ht]
\centering
\begin{pspicture}
(-0.5,-2.5)(11,3)
\psframe[linewidth=2pt,framearc=.3](-0.8,-2.5)(11.0,2.5)
\rput(0.0,2.0){\textcolor{black}{\large $\boldsymbol{K}$}}
\pscurve[linecolor=black,arrowscale=2]{->}(4,0)(3,-1)(1,-1)(0,0)
\pscurve[linecolor=black,arrowscale=2]{->}(0,0)(1,1)(3,1)(4,0)
\pscurve[linecolor=blue,arrowscale=2]{->}%
(1.00,0.5)(2.00,0.65)(3.0,0.25)(3.25,-0.25)(1.5,-0.55)(1.00,0.5)
\pspolygon*[linearc=0.2](2.00,-0.65)(2.5,0.15)(1.25,0.00)
\rput(1.9,-0.15){\textcolor{white}{\large $\boldsymbol{h}$}}
\psdots[dotstyle=o, linewidth=1.2pt,linecolor = green, fillcolor = red]%
(3.25,-0.25)(3.0,0.25)(1.00,0.5)(2.00,0.65)(2.25,-0.70)(1.5,-0.55)(3.25,-0.25)
\rput(2.0,-0.95){\textcolor{black}{\large $\boldsymbol{\cyc E_1}$}}
\rput(2.0,0.35){\textcolor{black}{\large $\boldsymbol{\cyc E_2}$}}
\psdots[dotstyle=o, linewidth=1.2pt,linecolor = black, fillcolor = yellow]%
(0,0)(1,-1)(1,1)(3,1)(3,-1)(4,0)
\rput(2.7,1.25){\textcolor{black}{\large $\boldsymbol{v_3}$}}\rput(1,1.25){\textcolor{black}{\large $\boldsymbol{v_2}$}}
\rput(-0.25,0){\textcolor{black}{\large $\boldsymbol{v_1}$}}\rput(3,-1.25){\textcolor{black}{\large $\boldsymbol{v_4}$}}
\rput(0.0,0.75){\textcolor{black}{\large $\boldsymbol{\vcyc E}$}}
\rput(1,-1.25){\textcolor{black}{\large $\boldsymbol{v_5}$}}
\rput(4.35,0.0){\textcolor{black}{\large $\boldsymbol{v_6}$}}
\pscurve[linecolor=black,arrowscale=2]{->}(10,0)(9,-1)(7,-1)(6,0)
\pscurve[linecolor=black,arrowscale=2]{->}(6,0)(7,1)(9,1)(10,0)
\pscurve[linecolor=blue,arrowscale=2]{->}(10,0)(9.0,0.75)(7.0,0.55)(6.8,0.0)(7.2,-0.5)(9.0,-0.5)(10,0)
\pspolygon*[linearc=0.2](8.00,-0.65)(8.5,0.15)(7.25,0.00)
\rput(7.9,-0.15){\textcolor{white}{\large $\boldsymbol{h'}$}}
\psdots[dotstyle=o, linewidth=1.2pt,linecolor = black, fillcolor = yellow]%
(6,0)(7,1.0)(9,-1)(7,-1)(10,0)(9,1)
\psdots[dotstyle=o, linewidth=1.2pt,linecolor = black, fillcolor = red]%
(10,0)(9.0,0.75)(7.0,0.55)(6.8,0.0)(7.2,-0.5)(9.0,-0.5)(10,0)
\rput(5.62,0){\textcolor{black}{\large $\boldsymbol{v_{10}}$}}\rput(7,1.25){\textcolor{black}{\large $\boldsymbol{v_{11}}$}}
\rput(9,1.25){\textcolor{black}{\large $\boldsymbol{v_{12}}$}}\rput(10.35,0.0){\textcolor{black}{\large $\boldsymbol{v_{13}}$}}
\rput(8.7,-1.265){\textcolor{black}{\large $\boldsymbol{v_{14}}$}}\rput(7,-1.25){\textcolor{black}{\large $\boldsymbol{v_{15}}$}}
\rput(6,0.75){\textcolor{black}{\large $\boldsymbol{\vcyc G}$}}
\rput(8.0,-0.95){\textcolor{black}{\large $\boldsymbol{\cyc G_1}$}}
\rput(8.0,0.55){\textcolor{black}{\large $\boldsymbol{\cyc G_2}$}}
\end{pspicture}
\caption[]{Pair of  Two Different Vortex Cycles With Holes}
\label{fig:2-vortexCycleHoles}
\end{figure}
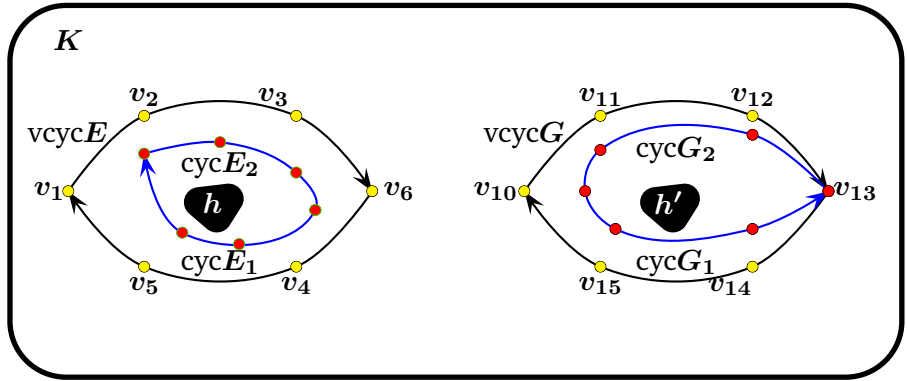

A darkened region in a planar shape represents a hole in the interior of the shape.   In cellular homology,  a cell complex $K$  is a Hausdorff space and a sequence of subspaces called skeletons~\cite{CookeFinney1967homologyCellComplex} (also called a CW complex or Closure-finite Weak topology complex~\cite{Hatcher2002CUPalgebraicTopology}).   Minimal planar skeletons are shown in Table~\ref{tab:skeleton}.

\begin{table}[!ht]\scriptsize 
\caption{Minimal Planar Cell Complex Skeletons}
\label{tab:skeleton}
\begin{tabular}{|c|c|c|c|}
    \hline
    Minimal Skeleton & $K_i, i = 0, 1, 1.5, 2$  & Planar Geometry & Interior\\
    \hline
    \hline
\begin{pspicture}
(0,0)(1,1)
\psdots[dotstyle=o, linewidth=1.2pt,linecolor = black, fillcolor = black]%
(0.5,0.5)	
\end{pspicture}		
& 
$K_0$
&
Vertex
&
nonempty
\\
    \hline
    \hline
\begin{pspicture}
(0,0)(1,1)
\psline[showpoints=true,linestyle=solid,linecolor = black]%
(0.25,0.25)(0.75,0.75)
\psdots[dotstyle=o, dotsize=1.3pt 2.25,linecolor = black, fillcolor = black]%
(0.25,0.25)(0.75,0.75)
\end{pspicture}& 
$K_1$
&
Line segment
&
nonempty
\\
    \hline
    \hline
\begin{pspicture}
(0,0)(1,1)
\psline*[linecolor = gray!55]%
(0.0,0.15)(0.85,0.85)(0.95,0.25)
\psline*[,linecolor = black, fillcolor = black]%
(0.3,0.25)(0.55,0.55)(0.85,0.3)
\psdots[dotstyle=o, dotsize=1.3pt 2.25,linecolor = black, fillcolor = black]%
(0.0,0.15)(0.85,0.85)(0.95,0.25)
\end{pspicture}& 
$K_{1.5}$
&
Partially filled triangle containing a 2-hole
&
nonempty\\
    \hline
    \hline
\begin{pspicture}
(0,0)(1,1)
\psline*[linecolor = gray!55]%
(0.0,0.15)(0.85,0.85)(0.95,0.25)
\psdots[dotstyle=o, dotsize=1.3pt 2.25,linecolor = black, fillcolor = black]%
(0.0,0.15)(0.85,0.85)(0.95,0.25)
\end{pspicture}& 
$K_2$
&
Filled triangle
&
nonempty
\\
\hline
\end{tabular}
\end{table}


Table~\ref{tab:skeleton} includes a $K_{1.5}$ skeleton, which is a filled triangle with a 2-hole in its interior.   The fractional dimension of a $K_{1.5}$ skeleton signals the fact such a skeleton has a partially filled interior, punctured with one or more holes.  A 2-\emph{hole} is a planar region with a boundary and an empty interior.   For example, a finite simple, closed curve that is the boundary of a planar shape defines a 2-hole.   

For a recent graphics study of polygons with holes in their interiors, see H. Boomari, M. Ostavari and A. Zarei~\cite{Boomari2018arXivPolygonsWithHoles}.   Also, from Table~\ref{tab:skeleton}, it is apparent from the grey shading that a $K_2$ skeleton is the intersection of three half planes that form a filled triangle.   Similarly, a 6-sided 1-cycle such as $\cyc A_2$ in vortex cycle $\vcyc A$ in Fig.~\ref{fig:1-vortexCycles} is the intersection of six half planes that construct a 6-gon with a nonempty interior.   Recall that a polytope that is the intersection of finitely-many closed half planes~\cite{Ziegler2007polytopes}.   In general, a 1-cycle is an $n$-sided polytope that is the intersection of $n$ half planes.

\begin{problem}
How many 2-holes are needed to destroy a 1-cycle, making it a shape boundary with an empty interior?
\end{problem}

\begin{problem}
The diameter of a 2-hole is the maximum distance between a pair  of points on the boundary of the 2-hole.   What is the diameter of a 2-hole in a filled, planar $n$-sided polytope that destroys a 1-cycle, making it a shape boundary with an empty interior?
\end{problem}

\begin{example}{\bf Vortex Cycles with Holes}.\\
Two different vortex cycles with holes are shown in Fig.~\ref{fig:2-vortexCycleHoles}, namely, $\vcyc E, \vcyc G$.    The vortex cycle $\vcyc E$ is an example of a 1-cycle within a 1-cycle ({\em i.e.}, $\cyc E_2$ within $\cyc E_1$) in which $\cyc E_2$ has a 2-hole $h$ in its interior.   The vortex cycle $\vcyc G$ is an example of intersecting 1-cycles ({\em i.e.}, $\cyc G_2$ within $\cyc G_1$) that form a vortex nerve in which $\cyc G_2$ has a 2-hole $h'$ in its interior.    In both cases, each inner 1-cycle is in the interior of an outer 1-cycle.   Hence, the 2-hole in the interior of the inner 1-cycle is common to the interiors of both 1-cycles in each vortex.   For example, 2-hole $h'$ in vortex nerve $\vcyc G$ is common to both of its 1-cycles.
\qquad \textcolor{blue}{\Squaresteel}
\end{example}

\begin{theorem}
Let $K$ be a collection of skeletons in a planar cell complex.
\begin{compactenum}[1$^o$]
\item In $K$, skeletons $K_0,K_1,K_2$ are planar shapes.
\item A $K_{1.5}$ skeleton is a planar shape.
\item A 1-cycle $cyc A$ with a hole $h\in \Int(cyc A)$ that is a proper subset in the interior of $cyc A$ is a planar shape.
\item A planar vortex cycle with a hole is a collection of overlapping 1-cycles, each with a hole.
\item A planar vortex cycle with a hole is a collection of concentric planar shapes.
\end{compactenum}
\end{theorem}
\begin{proof}$\mbox{}$\\
\begin{compactenum}[1$^o$:]
\item  By definition, every member of $K$ is a skeleton.   Each of the skeletons $K_0,K_1,K_2$ has a boundary
with nonempty interior.   Hence, these skeletons are planar shapes.
\item\label{lemma:skelHole} By definition, a $K_{1.5}$ skeleton is a closed 3-sided polytope that has a nonempty interior with a hole.   That is, let $h\in \Int(cyc A)$ be a 2-hole that is a proper subset in the interior of a $K_{1.5}$ skeleton.   In that case,  the nonempty part of interior of the $K_{1.5}$ skeleton $\Int(cyc A)$ equals $\Int(cyc A)\setminus h$.  In effect, $cyc A$ is a planar shape.
\item\label{lemma:1cycleHole} That a 1-cycle $cyc A$ with a hole that is a proper subset in the interior of $cyc A$ is a planar shape, follows from Part~\ref{lemma:skelHole}.
\item Immediate from Part~\ref{lemma:1cycleHole}.
\item Immediate from Part~\ref{lemma:1cycleHole} and Theorem~\ref{thm:vortexCycle}.
\end{compactenum}
\end{proof}

Let $(K,\dnear)$ be a collection of planar vortex cycles equipped with a descriptive proximity $\dnear$~\cite[\S 4]{DiConcilio2016descriptiveProximity},~\cite[\S 1.8]{Peters2016ComputationalProximity}, based on the descriptive intersection $\dcap$ of nonempty sets $A$ and $B$~\cite[\S 3]{Peters2013mcs}.   With respect to vortex cycles $\vcyc E, \vcyc G$ in $K$, for example, we consider $\vcyc E\dcap \vcyc G$, {\em i.e.}, the set of descriptions common to a pair of vortex cycles.   A vortex cycle description is a mapping $\Phi:2^K\longmapsto \mathbb{R}^n$ (an $n$-dimensional feature space).   For each given vortex cycle $\vcyc E$, find all vortex cycles $\vcyc G$ in $K$ that have nonempty descriptive intersection with $\vcyc E$, {\em i.e.}, $\cyc A\ \dcap\ \cyc B\neq \emptyset$ such that $\Phi(\vcyc G) = \Phi(\vcyc E)$.  This results in a Leader uniform topology on $H_1$~\cite{Leader1959}.   

\section{Preliminaries}
This section briefly presents the axioms for connectedness, strong and descriptive proximity.   A nonempty set $P$ is a proximity space, provided 
the closeness or remoteness of any two subsets in $P$ can be determined.   

\subsection{Cech Proximity Space}
A proximity space $P$ is sometimes called a $\delta$-space for $P$ equipped with~\cite{Smirnov1952OnProximitySpaces}, provided $P$ is equipped with a relation $\delta$ that satisfies, for example, the following \u{C}ech axioms for sets $A,B,C\in 2^P$.\\
\vspace{3mm}

\noindent {\bf \u{C}ech axioms}\cite[\S 2.5, p. 439]{Cech1966}
\begin{compactenum}[{P}1]
\item All subsets in $P$ are far from the empty set.
\item $A\ \delta\ B\ \implies B\ \delta\ A$, {\em i.e.}, $A$ close to $B$ implies $B$ is close to $A$.
\item $A\ \delta\ \left(B\cup C\right)\ \Leftrightarrow A\ \delta\ B\ \mbox{or}\ A\ \delta\ C$.
\item $A\cap B\neq \emptyset\ \implies A\ \delta\ B$.
\end{compactenum}
\mbox{}\\
\vspace{3mm}

A space $P$ equipped with the \u{C}ech proximity (denoted by $(P,\delta)$) is called a \u{C}ech proximity space.    We adopt the convention for a proximity metric $\delta:2^P\times 2^P\longrightarrow \{0,1\}$ introduced by Ju. M. Smirnov~\cite[\S 1, p. 8]{Smirnov1952OnProximitySpaces}.   We write $\delta(A,B) = 0$, provided subsets $A,B\in 2^P$ are close and $\delta(E,H) = 1$, provided subsets $A,B\in 2^P$ are not close, {\em i.e.}, there is a non-zero distance between $E$ and $H$.   Let $A,B,C\in 2^P$.   Then a proximity space satisfies the following properties.\\
\vspace{3mm}

\noindent {\bf Smirnov Proximity Space Properties}
\begin{compactenum}[{Q}1]
\item If $A\subseteq B$, then for any $C$, $\delta(A,C) \geq \delta(B,C)$.
\item Any sets which intersect are close.
\item No set is close to the empty set.
\end{compactenum}
\mbox{}\\
\vspace{3mm}

In a \u{C}ech proximity space, Smirnov proximity space property Q3 is satisfied by axiom $P1$ and property $Q2$ is satisfied by axioms P2-P4, {\em i.e.}, any subsets of P are close, provided the subsets have nonempty intersection.   That is, $A$ close to $B$ implies $B$ is close to $A$ (axiom P2).   Similarly, $A$ close to $B\cup C$ implies $A$ is close to $B$ or $A$ is close to $C$ (axiom P3) or $A$ is close to $B\cap C$ (axiom P4).  Let $A\cap C = \emptyset$.   Then $\delta(A,C) = 1$, since $A$ has no points in common with $C$.   Similarly, assume $B\cap C = \emptyset$.  Then,  $\delta(B,C) = 1$, since $B$ and $C$ have no points in common.   Hence, property Q1 is satisfied, since
\[ 
\delta(A,C) = \delta(B,C) = 1 \Rightarrow \delta(A,C) \geq \delta(B,C).
\]
For $A\subset B$ and $C\subset B$, we have $\delta(A,C) = 0$, since $A$ and $C$ have points in common.  Similarly, $\delta(B,C) = 0$.  Hence, $\delta(A,C) = \delta(B,C) = 0 \Rightarrow \delta(A,C) \geq \delta(B,C)$.

\subsection{Connectedness Proximity Space}
Let $K$ be a collection of skeletons in a planar cell complex and let $A,B,C$ be subsets containing skeletons in $K$ equipped with the relation $\mathop{\delta}\limits^{conn}$.   The pair $A,B$ is connected, provided $A\cap B\neq\emptyset$, {\em i.e.}, there is a skeleton in $A$ that has at least one vertex in common a skeleton in $B$.   Otherwise, $A$ and $B$ are disconnected.   

Let $X$ be a nonempty set and let $A,B\in 2^X$, nonempty subsets in the collection of subsets $2^X$.   $A$ and $B$ are mutually separated, provided $A\cap B = \emptyset$, {\em i.e.}, $A$ and $B$ have no points in common~\cite[\S 26.4, p. 192]{Willard1970}.   From the notion of separated sets, we obtain the following result for connected spaces.

\begin{theorem}\cite{Willard1970}$\mbox{}$\\
If $X = \mathop{\bigcup}\limits_{n-1}^{\infty} X_n$, where each $X_n\in 2^X$ is connected and $X_{n-1}\cap X_n\neq \emptyset$ for each $n\geq 2$, then space $X$ is connected.
\end{theorem}
\begin{proof}
The proof is given by S. Willard~\cite[\S 26.4, p. 193]{Willard1970}.   For a new kind of connectedness in which nonempty intersection is replaced by strong nearness, see C. Guadagni~\cite[p. 72]{Guadagni2015thesis} and in J.F. Peters~\cite[\S 1.16]{Peters2016ComputationalProximity}.   
\end{proof}

In this work, connectedness is defined in terms of the connectedness proximity $\conn$ and overlap connectedness $\sconn$ in Section~\ref{sec:overlapConnectedness}.   In both cases, nonempty intersection is replaced by a connectedness proximity in the study of connected cell complex spaces populated by connected skeletons.  For connected sets  $A,B\subset K$, we write $A\  \mathop{\delta}\limits^{conn}\ B$.   In effect, for each pair of skeletons $A,B$ in $K$,  $A\ \conn\ B$, provided there is a path between at least one vertex in $A$ and one or more vertices in $B$.   A \emph{path} is sequence of edges between a pair of vertices.  Equivalently, $A\cap B\neq \emptyset$ implies $A\  \mathop{\delta}\limits^{conn}\ B$.   If the sets of skeletons $A,B\in K$ are separated ({\em i.e.}, $A,B$ have no vertices in common), we write $A\ \mathop{\not\delta}\limits^{conn}\ B$.    This view of connectedness

Then the \u{C}ech axiom P4 is replaced by

\begin{description}
\item[{\bf P4conn}]  $A\cap B\neq \emptyset \ \Leftrightarrow A\ \mathop{\delta}\limits^{conn}\ B$.
\end{description}

\noindent By replacing $\delta$ with $\mathop{\delta}\limits^{conn}$ in the remaining \u{Cech} axioms, we obtain\\
\vspace{3mm}

\noindent {\bf Connectedness proximity axioms}.
\begin{compactenum}[{P}1{conn}]
\item $A\cap B = \emptyset \ \Leftrightarrow A\ \mathop{\not\delta}\limits^{conn}\ B$, {\em i.e.}, the sets of skeletons $A$ and $B$ are not close ($A$ and $B$ are far from each other).
\item $A\ \mathop{\delta}\limits^{conn}\ B\ \implies B\ \mathop{\delta}\limits^{conn}\ A$, {\em i.e.}, $A$ close to $B$ implies $B$ is close to $A$.
\item $A\ \mathop{\delta}\limits^{conn}\ \left(B\cup C\right)\ \implies A\ \mathop{\delta}\limits^{conn}\ B\ \mbox{or}\ A\ \mathop{\delta}\limits^{conn}\ C$.
\item $A\cap B\neq \emptyset \ \Leftrightarrow A\ \mathop{\delta}\limits^{conn}\ B$ (Connectedness Axiom).
\end{compactenum}
\mbox{}\\
\vspace{3mm}

\noindent A connectedness proximity space is denoted by $(K,\conn)$.   For $A,B\in K$, the Smirnov metric $\delta(A,B) = 0$ means that there is a path between any two vertices in $A\cup B$ and $\delta(A,B) = 1$ means that there is no path between any two vertices $A\cup B$.

\begin{lemma}\label{lemma:ConnectnessImpliesOverlap}
Let $K$ be a collection of skeletons in a planar cell complex equipped with the relation $\conn$.   Then $A\ \conn\ B$ implies $A\cap B\neq \emptyset$.
\end{lemma}
\begin{proof}
$A\ \conn\ B$, provided there is a path between any pair of vertices in skeletons $A$ and $B$, {\em i.e.}, $A,B$ are connected, provided there is a vertex common to $A$ and $B$.   Consequently, $A\cap B\neq \emptyset$.
\end{proof}

\begin{lemma}\label{lemma:skeletonProximitySpace}
Let $K$ be a connectedness space containing a collection of skeletons in a planar cell complex equipped with the relation $\conn$.   The space $K$ is a proximity space.
\end{lemma}
\begin{proof}
Let $A,B,C\in K$.  Smirnov proximity space property Q3 is satisfied by axiom $P1conn$ and property $Q2$ is satisfied by axioms {\bf P2conn}-{\bf P4conn}, {\em i.e.}, any sets of skeletons that are close, are connected.   Let $C\subset A\cup B$ ($C$ is part of the skeleton $A\cup B\in K$).  For any vertex $p$ in $A$ or $B$, there is a path between $p$ and any vertex $q\in C$.   Then $A\ \conn\ C$ and $B\ \conn\ C$.   Consequently, $\delta(A,C) = 0 = \delta(B,C)$,  Hence,  $\delta(A,C) \geq \delta(B,C)$.   If $(A\cup B)\cap C = \emptyset$ (the skeletons in $A$ and $B$ have no vertices in common with $C$), then
$\delta(A,C) = 1 = \delta(B,C)$ and $\delta(A,C) \geq \delta(B,C)$.   From axiom {\bf P4conn}, we have 
\[
(A\cup B) \mathop{\not\delta}\limits^{conn}\ C \Leftrightarrow (A\cup B)\cap C = \emptyset \Leftrightarrow \delta(A,C) = 1 = \delta(B,C) \Rightarrow \delta(A,C) \geq \delta(B,C).
\]
Smirnov property Q1 is satisfied.  Hence, $(K, \conn)$ is a proximity space.
\end{proof}

\begin{figure}[!ht]
\centering
\begin{pspicture}
(-1.5,-2.5)(11,3)
\psframe[linewidth=2pt,framearc=.3](-0.8,-2.5)(10.0,2.5)
\rput(0.0,2.0){\textcolor{black}{\large $\boldsymbol{K}$}}
\pscurve[linecolor=black,arrowscale=2]{->}(4,0)(3,-1)(1,-1)(0,0)
\pscurve[linecolor=black,arrowscale=2]{->}(0,0)(1,1)(3,1)(4,0)
\pscurve[linecolor=blue,arrowscale=2]{->}%
(1.00,0.5)(2.00,0.65)(3.0,0.25)(3.25,-0.25)(1.5,-0.55)(1.00,0.5)
\pspolygon*[linearc=0.2](2.00,-0.65)(2.5,0.15)(1.25,0.00)
\rput(1.9,-0.15){\textcolor{white}{\large $\boldsymbol{h}$}}
\psdots[dotstyle=o, linewidth=1.2pt,linecolor = green, fillcolor = red]%
(3.25,-0.25)(3.0,0.25)(1.00,0.5)(2.00,0.65)(2.25,-0.70)(1.5,-0.55)(3.25,-0.25)
\rput(2.0,-0.95){\textcolor{black}{\large $\boldsymbol{\cyc A_1}$}}
\rput(2.0,0.35){\textcolor{black}{\large $\boldsymbol{\cyc A_2}$}}
\psdots[dotstyle=o, linewidth=1.2pt,linecolor = black, fillcolor = yellow]%
(0,0)(1,-1)(1,1)(3,1)(3,-1)(4,0)
\rput(2.7,1.25){\textcolor{black}{\large $\boldsymbol{v_3}$}}\rput(1,1.25){\textcolor{black}{\large $\boldsymbol{v_2}$}}
\rput(-0.25,0){\textcolor{black}{\large $\boldsymbol{v_1}$}}\rput(3,-1.25){\textcolor{black}{\large $\boldsymbol{v_4}$}}
\rput(0.0,0.75){\textcolor{black}{\large $\boldsymbol{\vcyc A}$}}
\rput(1,-1.25){\textcolor{black}{\large $\boldsymbol{v_5}$}}
\rput(4.35,0.0){\textcolor{black}{\large $\boldsymbol{v_6}$}}
\psline[linecolor=black](5,0)(7,2)(9,-1)
\psdots[dotstyle=o, linewidth=1.2pt,linecolor = black, fillcolor = yellow]%
(5,0)(7,2)(9,-1)
\rput(5.25,-1.45){\textcolor{black}{\large $\boldsymbol{\skel E}$}}
\psline[linecolor=blue](4,0)(5,-2)(8,-1)
\psdots[dotstyle=o, linewidth=1.2pt,linecolor = green, fillcolor = red]%
(4,0)(5,-2)(8,-1)
\rput(8.25,1.05){\textcolor{black}{\large $\boldsymbol{\skel H}$}}
\end{pspicture}
\caption[]{Collection of Skeletons, including a Vortex Cycle with a Hole}
\label{fig:3-cellComplex}
\end{figure}
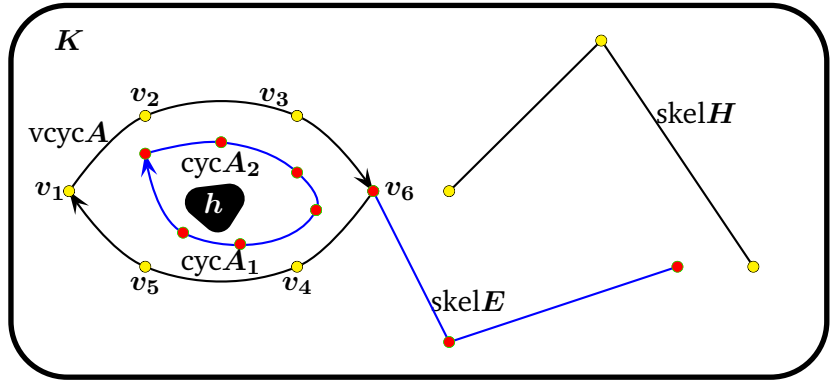


\begin{example}\label{ex:deltaConn} {\bf  Connectedness Proximity Space}.\\
Let $K$ be a collection of skeletons represented in Fig.~\ref{fig:3-cellComplex}, equipped with the proximity $\mathop{\delta}\limits^{conn}$.   A pair of skeletons in $K$ are close, provided the skeletons have at least one vertex in common.   For example, vortex cycle $\vcyc A$ and skeleton $\skel E$ have vertex $v_6$ in common.   Hence, from axiom {\bf P4conn}, we have
\[
v_6\in \vcyc A\ \cap\ \skel E \ \Leftrightarrow\ vcyc A\ \mathop{\delta}\limits^{conn}\ \skel E
\]
Skeletons that are not close have no vertices in common.   For example, in Fig.~\ref{fig:3-cellComplex},
\[
\skel E\ \mathop{\not\delta}\limits^{conn}\ \skel H,
\]
since the pair of skeletons $\skel E,\skel H$ have no vertices in common.   
\qquad \textcolor{blue}{\Squaresteel}
\end{example}

\begin{theorem}
Let $K$ be a collection of vortex cycles in a planar cell complex.   The space $K$ equipped with the relation $\mathop{\delta}\limits^{conn}$ is a proximity space.
\end{theorem}
\begin{proof}
A vortex cycle is a collection of concentric 1-cycles.   Each 1-cycle is a skeleton.   Then vortex cycle is a collection of skeletons and each collection of vortex cycles is also a collection of skeletons.   Hence, from Lemma~\ref{lemma:skeletonProximitySpace}, $K$ is a connectedness proximity space.
\end{proof}

\begin{figure}[!ht]
\centering
\begin{pspicture}
(-0.5,-1.5)(11,4.0)
\psframe[linewidth=2pt,framearc=.3](-0.8,-1.5)(11.0,4.25)
\rput(0.0,3.55){\textcolor{black}{\large $\boldsymbol{K}$}}
\pscurve[linecolor=black,arrowscale=2]{->}(4,0)(3,-1)(1,-1)(0,0)
\pscurve[linecolor=black,arrowscale=2]{->}(0,0)(1,1)(3,1)(4,0)
\pscurve[linecolor=blue,arrowscale=2]{->}%
(1.00,0.5)(1.50,0.85)(2.00,1.00)(3,1)(3.25,-0.25)(1.5,-0.55)(1.00,0.5)
\pscurve[linecolor=black,arrowscale=2]{->}%
(3,1)(3.00,2.55)(4.0,3.55)(5.0,3.55)(5.95,2.55)(4.0,1.00)(3,1)
\psdots[dotstyle=o, linewidth=1.2pt,linecolor = black, fillcolor = yellow]%
(3,1)(3.00,2.55)(4.0,3.55)(5.0,3.55)(5.95,2.55)(4.0,1.00)(3,1)
\rput(2.65,3.0){\textcolor{black}{\large $\boldsymbol{\vNrv E}$}}
\rput(4.5,3.25){\textcolor{black}{\large $\boldsymbol{\cyc E_1}$}}
\pscurve[linecolor=blue,arrowscale=2]{->}%
(3,1)(3.55,2.55)(4.0,3.0)(5.0,3.0)(5.5,2.5)(5.0,1.80)(3,1)
\psdots[dotstyle=o, linewidth=1.2pt,linecolor = green, fillcolor = red]%
(3,1)(3.55,2.55)(4.0,3.0)(5.0,3.0)(5.5,2.5)(5.0,1.80)(3,1)
\rput(4.5,2.0){\textcolor{black}{\large $\boldsymbol{\cyc E_2}$}}
\psdots[dotstyle=o, linewidth=1.2pt,linecolor = green, fillcolor = red]%
(3.25,-0.25)(3.0,1)(1.00,0.5)(2.00,1.00)(2.25,-0.70)(1.5,0.85)(3.25,-0.25)
\rput(2.0,-0.95){\textcolor{black}{\large $\boldsymbol{\cyc A_1}$}}
\rput(2.0,0.35){\textcolor{black}{\large $\boldsymbol{\cyc A_2}$}}
\psdots[dotstyle=o, linewidth=1.2pt,linecolor = black, fillcolor = yellow]%
(0,0)(1,-1)(1,1)(3,-1)(4,0)
\rput(2.7,1.25){\textcolor{black}{\large $\boldsymbol{v_3}$}}
\rput(-0.25,0){\textcolor{black}{\large $\boldsymbol{v_1}$}}
\rput(0.0,0.75){\textcolor{black}{\large $\boldsymbol{\vNrv A}$}}
\rput(4.35,0.0){\textcolor{black}{\large $\boldsymbol{v_6}$}}
\pscurve[linecolor=black,arrowscale=2]{->}(10,0)(9,-1)(7,-1)(6,0)
\pscurve[linecolor=black,arrowscale=2]{->}(6,0)(7,3.5)(9,3.5)(10,0)
\pscurve[linecolor=blue,arrowscale=2]{->}%
(10,0)(9.0,3.0)(7.0,3.0)(6.8,0.0)(7.2,-0.5)(9.0,-0.5)(10,0)
\psdots[dotstyle=o, linewidth=1.2pt,linecolor = black, fillcolor = yellow]%
(6,0)(7,3.5)(9,-1)(7,-1)(10,0)(9,3.5)
\psdots[dotstyle=o, linewidth=1.2pt,linecolor = black, fillcolor = red]%
(10,0)(9.0,3.0)(7.0,3.0)(6.8,0.0)(7.2,-0.5)(9.0,-0.5)(10,0)
%
\pscurve[linecolor=black,arrowscale=2]{->}(7,1)(7.5,2.5)(8.5,2.75)(9.25,2)(8.5,0)(7.5,0)(7,1)
\psdots[dotstyle=o, linewidth=1.2pt,linecolor = black, fillcolor = yellow]%
(7,1)(7.5,2.5)(8.5,2.75)(9.25,2)(8.5,0)(7.5,0)(7,1)
\pscurve[linecolor=black,arrowscale=2]{->}(7,1)(7.75,2.25)(8.25,2.25)(8.75,2)(8.25,0.5)(7.75,0.5)(7,1)
\psdots[dotstyle=o, linewidth=1.2pt,linecolor = black, fillcolor = red]%
(7,1)(7.75,2.25)(8.25,2.25)(8.68,1)(8.75,2)(8.25,0.5)(7.75,0.5)(7,1)
\rput(7.25,1.55){\textcolor{black}{\large $\boldsymbol{\vNrv H}$}}
\rput(8.0,-0.25){\textcolor{black}{\large $\boldsymbol{\cyc H_1}$}}
\rput(8.0,2.0){\textcolor{black}{\large $\boldsymbol{\cyc H_2}$}}
\rput(5.62,0){\textcolor{black}{\large $\boldsymbol{v_{10}}$}}
\rput(10.35,0.0){\textcolor{black}{\large $\boldsymbol{v_{13}}$}}
\rput(6.15,0.75){\textcolor{black}{\large $\boldsymbol{\vNrv B}$}}
\rput(8.0,-1.25){\textcolor{black}{\large $\boldsymbol{\cyc B_1}$}}
\rput(8.0,3.35){\textcolor{black}{\large $\boldsymbol{\cyc B_2}$}}
\end{pspicture}
\caption[]{Collection of Proximal Vortex Nerves}
\label{fig:4-vortexNerves}
\end{figure}
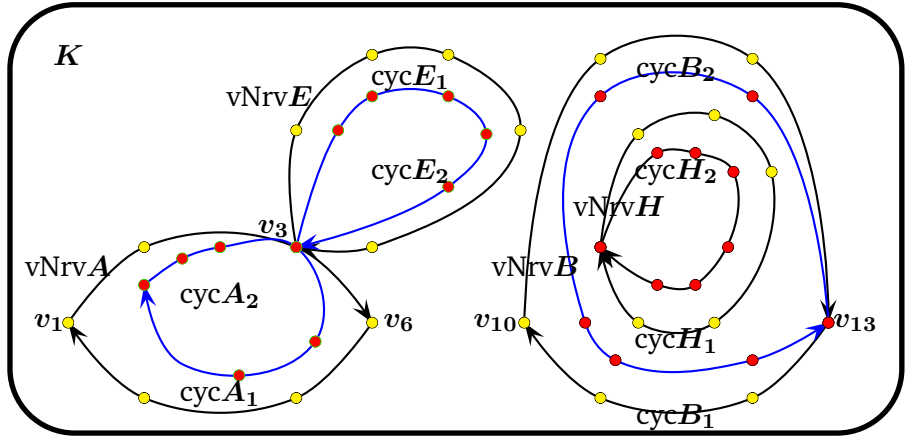

\setlength{\intextsep}{0pt}
\begin{wrapfigure}[8]{R}{0.42\textwidth}
\begin{minipage}{4.0 cm}
\centering
\begin{pspicture}
(0,-1.0)(4,1)
\pscurve[linecolor=black,arrowscale=2]{->}%
(1,0)(2,1)(3,1)(4,0)(3,-1)(2,-1)(1,0)
\psdots[dotstyle=o, linewidth=1.2pt,linecolor = black, fillcolor = yellow]%
(1,0)(2,1)(3,1)(4,0)(3,-1)(2,-1)(1,0)
\pscurve[linecolor=blue,arrowscale=2]{->}%
(1,0)(2,0.5)(3,0.5)(3.5,0)(3,-0.5)(2,-0.5)(1,0)
\psdots[dotstyle=o, linewidth=1.2pt,linecolor = green, fillcolor = red]%
(1,0)(2,0.5)(3,0.5)(3.5,0)(3,-0.5)(2,-0.5)(1,0)
\rput(2.5,-0.75){\textcolor{black}{\large $\boldsymbol{\cyc A_1}$}}
\rput(2.5,-0.25){\textcolor{black}{\large $\boldsymbol{\cyc A_2}$}}
\pspolygon*[linearc=0.2](+0.75,-0.25)(1.35,1.45)(1.25,-1.00)
\psdots[dotstyle=o, linewidth=1.2pt,linecolor = red, fillcolor = green]%
(1.25,0.25)(1.25,-0.30)
\rput(1.25,0.45){\textcolor{white}{\large $\boldsymbol{v}$}}
\rput(1.25,-0.40){\textcolor{white}{\large $\boldsymbol{v'}$}}
\rput(1,0){\textcolor{white}{\large $\boldsymbol{h}$}}
\end{pspicture}
\caption[]{$\boldsymbol{\left(\cyc A_1\cup \cyc A_2\right)\ \cap\ h}$}
\label{fig:cyclesOnHole}
\end{minipage}
\end{wrapfigure}
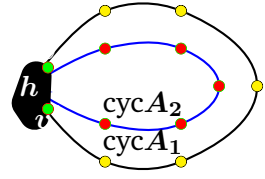 

\subsection{Vortex Nerves Proximity space}
A vortex cycle $\vcyc A$ containing 1-cycles with a common vertex is an example of a vortex nerve (denoted by $vNrv A$).   A collection of vortex nerves equipped with the $\conn$ proximity is a connectedness  proximity space.

\begin{theorem}\label{thm:vortexNerveSpace}
Let $K$ be a collection of vortex nerves in a planar cell complex.   The space $K$ equipped with the relation $\mathop{\delta}\limits^{conn}$ is a proximity space.
\end{theorem}
\begin{proof}
Each vortex nerve is a collection of intersecting 1-cycles, which are skeletons.   The results follows from Lemma~\ref{lemma:skeletonProximitySpace}, since $K$ is also a collection of skeletons equipped with the proximity $\conn$.
\end{proof}

\begin{example} {\bf Vortex Nerves Proximity Space}.\\
Three vortex nerves $\vNrv A,\vNrv E\ \mbox{attached to $\vNrv A$}, \vNrv B, \vNrv H\ \mbox{in the interior of $\vNrv B$}$ in a cell complex $K$ are represented in Fig.~\ref{fig:4-vortexNerves}.   Let the collection of vortex nerves $K$ be equipped with the proximity $\conn$.    Vortex nerves are close, provided the nerves have nonempty intersection.  For example, $\vNrv A\ \conn\ \vNrv E$, {\em i.e.}, $\delta(\vNrv A, \vNrv E) = 0$.   Hence, Smirnov property Q2 is satisfied by $\left(K,\conn\right)$.   Vortex nerves are far (not close), provided the vortex nerves have empty intersection.   For example,  $\vNrv A\ \farconn\ \vNrv E$, {\em i.e.}, $\delta(\vNrv A, \vNrv E) = 1$ (Smirnov property Q3).    We also have, for example,
\begin{align*}
\delta(\vNrv A,\vNrv H) &= 1 = \delta(\vNrv B,\vNrv H)\ \mbox{non-intersecting nerves are far},\\
\delta(\vNrv H,\vNrv E) &= 1\ \mbox{and}\  \delta(\vNrv A,\vNrv E) = 0\\
 &\Leftrightarrow \delta(\vNrv H,\vNrv E) \geq \delta(\vNrv A,\vNrv E).
\end{align*}
In effect, Smirnov property Q1 is satisfied.   Hence, $\left(K,\conn\right)$ is a connectedness proximity space.
\qquad \textcolor{blue}{\Squaresteel}
\end{example}

\begin{example} {\bf Spacetime Vortex Nerves Proximity Space}.\\
Spacetime vortex nerves (overlapping vortex cycles) have been observed in recent studies of ground vortex aerodynamics by J.P. Murphy and D.G. MacManus~\cite{Murphy2011ground} and  in the vortex flows of overlapping jet streams in ground proximity by J.M.M. Barata, N. Bernardo, P.J.C.T. Santos and A.R.R. Silva~\cite{Silva2011vortexFlows} and by A.R.R. Silva, D.F.G. Dur\~{a}o, J.M.M. Barata, P. Santos  S. Ribeiro~\cite{Silva2008groundVortexFlow}.   Physical vortex nerves can be observed in the representation of the contours of overlapping turbulence velocity vortices in, for example, Figure 6 in ~\cite[p. 8]{Silva2008groundVortexFlow} and vortex pairs systems in Figure 7 in P.R. Spalart, M. Kh. Strelets, A.K. Travin and M.L. Slur~\cite{Spalart2001vortexPairs}.
\qquad \textcolor{blue}{\Squaresteel}
\end{example}

The presence of holes in the interiors of vortex nerves in a cell complex equipped with the proximity $\conn$ gives us the following result.

\begin{corollary}
Let $K$ be a collection of vortex nerves containing holes in their interiors in a planar cell complex.   The space $K$ equipped with the relation $\mathop{\delta}\limits^{conn}$ is a proximity space.
\end{corollary}
\begin{proof}
Immediate from Theorem~\ref{thm:vortexNerveSpace}, since the relationships between vortex nerves in $K$ is unaffected by the presence of holes in the interiors of the nerves.
\end{proof}

\begin{example}
A pair of disjoint vortex nerves containing holes in their interiors is represented in Fig.~\ref{fig:5-vortexNerveHoles}.
\qquad \textcolor{blue}{\Squaresteel}
\end{example}

\begin{figure}[!ht]
\centering
\begin{pspicture}
(-0.5,-2.5)(11,3)
\psframe[linewidth=2pt,framearc=.3](-0.8,-2.5)(11.0,2.5)
\rput(0.0,2.0){\textcolor{black}{\large $\boldsymbol{K}$}}
\pscurve[linecolor=black,arrowscale=2]{->}(4,0)(3,-1)(1,-1)(0,0)
\pscurve[linecolor=black,arrowscale=2]{->}(0,0)(1,1)(3,1)(4,0)
\pscurve[linecolor=blue,arrowscale=2]{->}%
(1.00,0.5)(2.00,0.65)(3.0,0.25)(4.0,0.0)(3.0,-0.25)(1.5,-0.55)(1.00,0.5)
\pspolygon*[linearc=0.2](2.00,-0.65)(2.5,0.15)(1.25,0.00)
\rput(1.9,-0.15){\textcolor{white}{\large $\boldsymbol{h}$}}
\psdots[dotstyle=o, linewidth=1.2pt,linecolor = green, fillcolor = red]%
(1.00,0.5)(2.00,0.65)(3.0,0.25)(4.0,0.0)(3.0,-0.25)(1.5,-0.55)(1.00,0.5)
\rput(2.0,-0.95){\textcolor{black}{\large $\boldsymbol{\cyc E_1}$}}
\rput(2.0,0.35){\textcolor{black}{\large $\boldsymbol{\cyc E_2}$}}
\psdots[dotstyle=o, linewidth=1.2pt,linecolor = black, fillcolor = yellow]%
(0,0)(1,-1)(1,1)(3,1)(3,-1)(4,0)
\rput(2.7,1.25){\textcolor{black}{\large $\boldsymbol{v_3}$}}\rput(1,1.25){\textcolor{black}{\large $\boldsymbol{v_2}$}}
\rput(-0.25,0){\textcolor{black}{\large $\boldsymbol{v_1}$}}\rput(3,-1.25){\textcolor{black}{\large $\boldsymbol{v_4}$}}
\rput(0.0,0.75){\textcolor{black}{\large $\boldsymbol{\vcyc E}$}}
\rput(1,-1.25){\textcolor{black}{\large $\boldsymbol{v_5}$}}
\rput(4.35,0.0){\textcolor{black}{\large $\boldsymbol{v_6}$}}
\pscurve[linecolor=black,arrowscale=2]{->}(10,0)(9,-1)(7,-1)(6,0)
\pscurve[linecolor=black,arrowscale=2]{->}(6,0)(7,1)(9,1)(10,0)
\pscurve[linecolor=blue,arrowscale=2]{->}(10,0)(9.0,0.75)(7.0,0.55)(6.8,0.0)(7.2,-0.5)(9.0,-0.5)(10,0)
\pspolygon*[linearc=0.2](8.00,-0.65)(8.5,0.15)(7.25,0.00)
\rput(7.9,-0.15){\textcolor{white}{\large $\boldsymbol{h'}$}}
\psdots[dotstyle=o, linewidth=1.2pt,linecolor = black, fillcolor = yellow]%
(6,0)(7,1.0)(9,-1)(7,-1)(10,0)(9,1)
\psdots[dotstyle=o, linewidth=1.2pt,linecolor = black, fillcolor = red]%
(10,0)(9.0,0.75)(7.0,0.55)(6.8,0.0)(7.2,-0.5)(9.0,-0.5)(10,0)
\rput(5.62,0){\textcolor{black}{\large $\boldsymbol{v_{10}}$}}\rput(7,1.25){\textcolor{black}{\large $\boldsymbol{v_{11}}$}}
\rput(9,1.25){\textcolor{black}{\large $\boldsymbol{v_{12}}$}}\rput(10.35,0.0){\textcolor{black}{\large $\boldsymbol{v_{13}}$}}
\rput(8.7,-1.265){\textcolor{black}{\large $\boldsymbol{v_{14}}$}}\rput(7,-1.25){\textcolor{black}{\large $\boldsymbol{v_{15}}$}}
\rput(6,0.75){\textcolor{black}{\large $\boldsymbol{\vcyc G}$}}
\rput(8.0,-0.95){\textcolor{black}{\large $\boldsymbol{\cyc G_1}$}}
\rput(8.0,0.55){\textcolor{black}{\large $\boldsymbol{\cyc G_2}$}}
\end{pspicture}
\caption[]{Pair of  Disjoint Vortex Nerves With Holes}
\label{fig:5-vortexNerveHoles}
\end{figure}
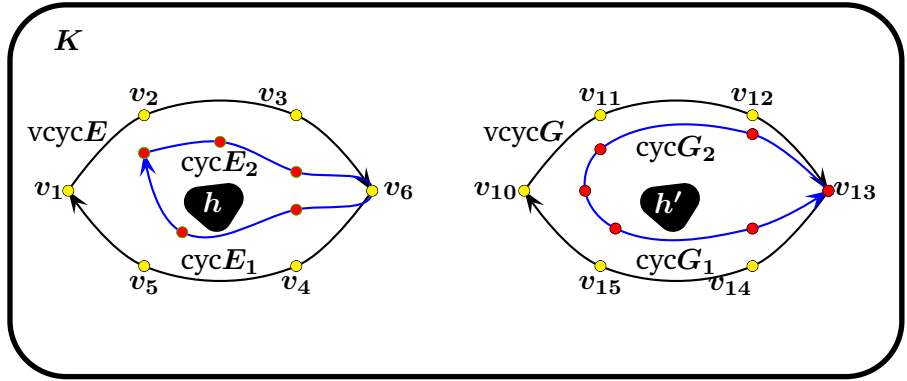

\begin{problem}
Let $K$ be a collection of vortex nerves so that the boundary of each of the holes has more than one vertex that is in  the intersection 1-cycles in each of the nerves in a planar cell complex.    For an example of vortex cycles that overlap vertices on the boundary of a hole, see Fig.~\ref{fig:cyclesOnHole}.    Prove that a vortex nerve is destroyed by a hole whose boundary overlaps the nerve cycles in more than one vertex.
\end{problem}

\begin{problem}
Let $K$ be a collection of vortex nerves so that the boundary of each of the holes has a single vertex that is in the intersection of the 1-cycles in each of the nerves in a planar cell complex.    Also let $K$ be equipped the proximity $\conn$.   Prove that $K$ is a connectedness proximity space.
\end{problem}

\setlength{\intextsep}{0pt}
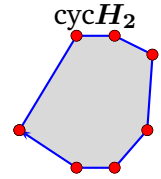
\begin{wrapfigure}[13]{R}{0.35\textwidth}
\begin{minipage}{3.5 cm}
\centering
\begin{pspicture}
(1,0.5)(4,3)
\psline*[linecolor=gray!30]{->}%
(2,1)(2.75,2.25)(3.25,2.25)(3.75,2)(3.68,1)(3.25,0.5)(2.75,0.5)(2,1)
\psline[linecolor=blue]{->}%
(2,1)(2.75,2.25)(3.25,2.25)(3.75,2)(3.68,1)(3.25,0.5)(2.75,0.5)(2,1)
\psdots[dotstyle=o, linewidth=1.2pt,linecolor = black, fillcolor = red]%
(2,1)(2.75,2.25)(3.25,2.25)(3.75,2)(3.68,1)(3.25,0.5)(2.75,0.5)(2,1)
\rput(3.0,2.5){\textcolor{black}{\large $\boldsymbol{\cyc H_2}$}}
\end{pspicture}
\caption[]{$\in \vNrv H$}
\label{fig:cycleInterior}
\end{minipage}
\end{wrapfigure} 

\subsection {Neighbourhoods, Set Closure, Boundary, Interior and CW Topology}
The interior of a nonempty set is considered, here.   It is the interior of a vortex cycle that leads to strong forms of connectedness proximity on a shapes in cell complex in which the interiors of vortices overlap either spatially or descriptively.   Let $A$ be a nonempty set of vertices, $p\in A$ in a bounded region $X$ of the Euclidean plane.  An \emph{open ball} $B_r(p)$ with radius $r$ is defined by
\[
B_r(p) = \left\{q\in X: \norm{p - q} < r\right\}.
\]
The \emph{closure} of $A$ (denoted by $\cl A$) is defined by
\[
\cl A = \left\{q\in X: B_r(q)\subset A\ \mbox{for some $r$}\right\}\ \mbox{(Closure of set $A$)}.
\]
The \emph{boundary} of $A$ (denoted by $\bdy A$) is defined by
\[
\bdy A = \left\{q\in X: B(q)\subset A\ \cap\ X\setminus A\right\}\ \mbox{(Boundary of set $A$)}.
\]
Of great interest in the study of the closeness of vortex cycles is the interior of a shape, found by subtracting the boundary of a shape from its closure.  In general, the \emph{interior} of a nonempty set $A\subset X$ (denoted by $\Int A$) defined by
\[
\Int A = \cl A - \bdy A\ \mbox{(Interior of set $A$)}.
\]

Let the cell complex $K$ be a Hausdorff space.   Let $A$ be a cell (skeleton) in $K$.   Each cell decomposition $A,B\in K$ is called a CW complex, provided
\begin{description}
\item [{\bf Closure Finiteness}] Closure of every cell (skeleton) $\cl A$ intersects on a finite number of other cells.
\item [{\bf Weak topology}]  $A\in 2^K$ is closed ($A = \bdy A\cup \Int A$), provided $A\cap \cl B$ is closed, {\em i.e.},\\
$
A\cap \cl B = \bdy(A\cap \cl B)\cup Int(A\cap \cl B).
$
\end{description}
$K$ has a topology $\tau$ that is a CW topology~\cite{Whitehead1949BAMS-CWtopology},~\cite[\S 2.4, p. 81]{Peters2017AMSJshapeSignature}, provided $\tau$ has the closure finiteness and weak topology properties.

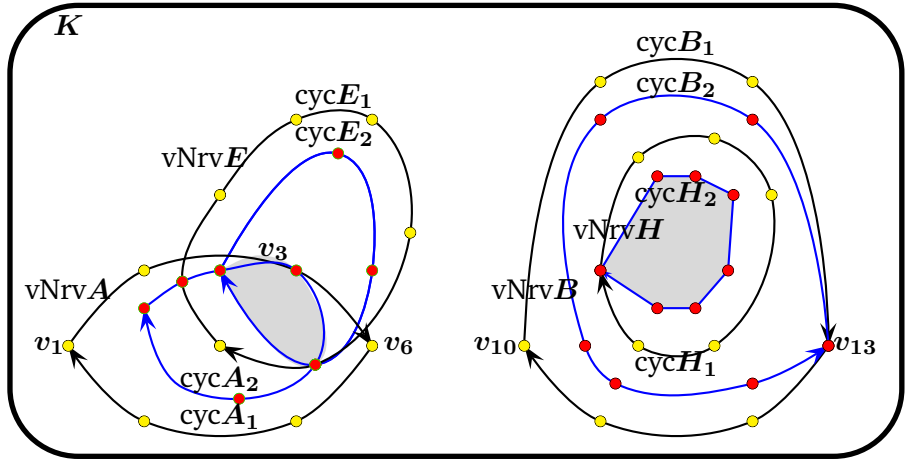
\begin{figure}[!ht]
\centering
\begin{pspicture}
(-0.5,-1.5)(11,5.0)
\psframe[linewidth=2pt,framearc=.3](-0.8,-1.5)(11.0,4.55)
\rput(0.0,4.25){\textcolor{black}{\large $\boldsymbol{K}$}}
\pscurve[linecolor=blue]{->}%
(2,1)(3.55,2.55)(4,1)(3.25,-0.25)(2,1)
\pscurve*[linecolor=gray!30]{->}%
(2,1)(2.5,1.2)(3,1)(3.25,-0.25)(2,1)
\pscurve[linecolor=blue,arrowscale=2]{->}%
(2,1)(3.55,2.55)(4,1)(3.25,-0.25)(2,1)
\pscurve[linecolor=black,arrowscale=2]{->}(4,0)(3,-1)(1,-1)(0,0)
\pscurve[linecolor=black,arrowscale=2]{->}(0,0)(1,1)(3,1)(4,0)
\pscurve[linecolor=blue,arrowscale=2]{->}%
(1.00,0.5)(1.50,0.85)(2.00,1.00)(3,1)(3.25,-0.25)(1.5,-0.55)(1.00,0.5)
\pscurve[linecolor=black,arrowscale=2]{->}%
(2,0)(1.50,0.85)(2,2)(3,3)(4,3)(4.5,1.5)(3.25,-0.25)(2,0)
\psdots[dotstyle=o, linewidth=1.2pt,linecolor = green, fillcolor = red]%
(2,1)(3.55,2.55)(4,1)(3.25,-0.25)
\psdots[dotstyle=o, linewidth=1.2pt,linecolor = black, fillcolor = yellow]%
(2,0)(1.50,0.85)(2,2)(3,3)(4,3)(4.5,1.5)(3.25,-0.25)(2,0)
\rput(3.5,2.8){\textcolor{black}{\large $\boldsymbol{\cyc E_2}$}}
\rput(3.5,3.3){\textcolor{black}{\large $\boldsymbol{\cyc E_1}$}}
\rput(1.8,2.5){\textcolor{black}{\large $\boldsymbol{\vNrv E}$}}
\psdots[dotstyle=o, linewidth=1.2pt,linecolor = green, fillcolor = red]%
(3.25,-0.25)(3.0,1)(1.00,0.5)(2.00,1.00)(2.25,-0.70)(1.5,0.85)(3.25,-0.25)
\rput(2.0,-0.95){\textcolor{black}{\large $\boldsymbol{\cyc A_1}$}}
\rput(2.0,-0.45){\textcolor{black}{\large $\boldsymbol{\cyc A_2}$}}
\psdots[dotstyle=o, linewidth=1.2pt,linecolor = black, fillcolor = yellow]%
(0,0)(1,-1)(1,1)(3,-1)(4,0)
\rput(2.7,1.25){\textcolor{black}{\large $\boldsymbol{v_3}$}}
\rput(-0.25,0){\textcolor{black}{\large $\boldsymbol{v_1}$}}
\rput(0.0,0.75){\textcolor{black}{\large $\boldsymbol{\vNrv A}$}}
\rput(4.35,0.0){\textcolor{black}{\large $\boldsymbol{v_6}$}}
\pscurve[linecolor=black,arrowscale=2]{->}(10,0)(9,-1)(7,-1)(6,0)
\pscurve[linecolor=black,arrowscale=2]{->}(6,0)(7,3.5)(9,3.5)(10,0)
\pscurve[linecolor=blue,arrowscale=2]{->}%
(10,0)(9.0,3.0)(7.0,3.0)(6.8,0.0)(7.2,-0.5)(9.0,-0.5)(10,0)
\psdots[dotstyle=o, linewidth=1.2pt,linecolor = black, fillcolor = yellow]%
(6,0)(7,3.5)(9,-1)(7,-1)(10,0)(9,3.5)
\psdots[dotstyle=o, linewidth=1.2pt,linecolor = black, fillcolor = red]%
(10,0)(9.0,3.0)(7.0,3.0)(6.8,0.0)(7.2,-0.5)(9.0,-0.5)(10,0)
%
\pscurve[linecolor=black,arrowscale=2]{->}(7,1)(7.5,2.5)(8.5,2.75)(9.25,2)(8.5,0)(7.5,0)(7,1)
\psdots[dotstyle=o, linewidth=1.2pt,linecolor = black, fillcolor = yellow]%
(7,1)(7.5,2.5)(8.5,2.75)(9.25,2)(8.5,0)(7.5,0)(7,1)
\psline*[linecolor=gray!30]{->}%
(7,1)(7.75,2.25)(8.25,2.25)(8.75,2)(8.68,1)(8.25,0.5)(7.75,0.5)(7,1)
\psline[linecolor=blue]{->}%
(7,1)(7.75,2.25)(8.25,2.25)(8.75,2)(8.68,1)(8.25,0.5)(7.75,0.5)(7,1)
\psdots[dotstyle=o, linewidth=1.2pt,linecolor = black, fillcolor = red]%
(7,1)(7.75,2.25)(8.25,2.25)(8.68,1)(8.75,2)(8.25,0.5)(7.75,0.5)(7,1)
\rput(7.25,1.55){\textcolor{black}{\large $\boldsymbol{\vNrv H}$}}
\rput(8.0,-0.25){\textcolor{black}{\large $\boldsymbol{\cyc H_1}$}}
\rput(8.0,2.0){\textcolor{black}{\large $\boldsymbol{\cyc H_2}$}}
\rput(5.62,0){\textcolor{black}{\large $\boldsymbol{v_{10}}$}}
\rput(10.35,0.0){\textcolor{black}{\large $\boldsymbol{v_{13}}$}}
\rput(6.15,0.75){\textcolor{black}{\large $\boldsymbol{\vNrv B}$}}
\rput(8.0,4.0){\textcolor{black}{\large $\boldsymbol{\cyc B_1}$}}
\rput(8.0,3.45){\textcolor{black}{\large $\boldsymbol{\cyc B_2}$}}
\end{pspicture}
\caption[]{Vortex Nerves with Overlapping Interiors}
\label{fig:7-overlappingVortexNerves}
\end{figure}

\subsection{Overlap Connectedness Proximity Space}\label{sec:overlapConnectedness}
In this section, weak and strong connectedness proximities of skeletons arise when we consider pairs of vortex cycles with overlapping interiors.   Let $K$ be a collection of vortex cycles equipped with the proximity $\sconn$, which is a form of the strong proximity $\sn$~\cite[\S 1.9, pp. 28-30]{Peters2016ComputationalProximity}.   The weak and strong forms of $\sconn$ satisfy the following axioms.\\
\vspace{3mm}
\begin{description}
\item[{\bf P4overlap [weak option]}]  $\Int A\ \cap\ \Int B\neq \emptyset\ \Rightarrow\ A\ \sconn\ B$.
\item[{\bf P5overlap [strong option]}] $A\ \sconn\ B \Rightarrow A\cap B\neq \emptyset$
\end{description}
\mbox{}\\
\vspace{3mm}

Axiom {\bf P4overlap} is a rewrite of the \u{C}ech axiom {\bf P4} and axiom {\bf P5overlap} is addition to the usual \u{C}ech axioms.
It is easy to see that $\sn$ satisfies the remaining \u{C}ech axioms after replacing $\delta$ with $\sn$.   Let $A,B,C\in K$, a cell complex space equipped with the proximity $\sconn$, which satisfies the following axioms.\\
\vspace{3mm}

\noindent {\bf Overlap Connectedness proximity axioms}.
\begin{compactenum}[{P}1{intConn}]
\item $A\cap B = \emptyset \ \Leftrightarrow A\ \sconn\ B$, {\em i.e.}, the sets of skeletons $A$ and $B$ are not close ($A$ and $B$ are far from each other).
\item $A\ \sconn\ B\ \implies B\ \sconn\ A$, {\em i.e.}, $A$ overlaps (is close to) $B$ implies $B$ overlaps (is close to) $A$.
\item $A\ \sconn\ \left(B\cup C\right)\ \implies A\ \sconn\ B\ \mbox{or}\ A\ \sconn\ C$.
\item $\Int A\ \cap\ \Int B\neq \emptyset\ \Rightarrow\ A\ \sconn\ B$ (Weak Overlap Connectedness Axiom).
\item  $A\ \sconn\ B \Rightarrow A\cap B\neq \emptyset$  (Strong Overlap Connectedness Axiom). 
\qquad \textcolor{blue}{\Squaresteel}
\end{compactenum}
\mbox{}\\
\vspace{3mm}

An overlap connectedness space is denoted by $\left(K,\sconn\right)$.   Skeletons $A,B$ in $K$ are close, provided the interior $\Int A$ has nonempty intersection with the interior $\Int A$.

\begin{theorem}\label{thm:overlapSkeletonSpace}
Let $K$ be a collection of vortex nerves in a planar cell complex.   The space $K$ equipped with the relation $\sconn$ is a proximity space.
\end{theorem}
\begin{proof}
The result follows from Lemma~\ref{lemma:skeletonProximitySpace}, since $K$ is also a collection of skeletons equipped with the proximity $\conn$.
\end{proof}

\begin{example}{\bf Overlapping Vortex Nerves}.\\
Two pairs of overlapping vortex nerves are represented in Fig.~\ref{fig:7-overlappingVortexNerves}, namely, $\vNrv A\ \sconn\ \vNrv E$ and $\vNrv B\ \sconn\ \vNrv H$.    In the case of the pair of vortex nerves $vNrv A,vNrv E$, the gray region for these nerves in Fig.~\ref{fig:7-overlappingVortexNerves} represents the nonempty intersection of the interior of the 1-cycle $\Int \cyc A_2\in \vNrv A$ and the interior of the 1-cycle $\Int \cyc E_2\in \vNrv E$.   From axiom P4intConn, we have
\begin{align*}
\Int \cyc A_2\ \cap\ \Int \cyc E_2\neq \emptyset\ &\Rightarrow\ \cyc A_2\ \sconn\ \cyc E_2\\
                                                    &\Rightarrow\ \vNrv A\ \sconn\ \vNrv E,\ \mbox{Axiom P5intConn, we have}\\
\vNrv A\ \sconn\ \vNrv E &\Rightarrow\ \Int \cyc A_2\ \cap\ \Int \cyc E_2\neq \emptyset.
\end{align*}
Concentric vortex nerves $\vNrv B,\vNrv H$ are also represented in Fig.~\ref{fig:7-overlappingVortexNerves},  The interior $Int \cyc H_2$ is represented in Fig.~\ref{fig:cycleInterior}in the vortex nerve $\vNrv H$, which is in the interior of vortex nerve $\vNrv B$.   Again, from axiom P4intConn, we have
\begin{align*}
\Int \vNrv\ B\ \cap\ \Int\ \vNrv\ H\neq \emptyset\ &\Rightarrow\ \vNrv\ B\ \sconn\ \vNrv\ H,\ \mbox{and from Axiom P5intConn, we have}\\
\vNrv B\ \sconn\ \vNrv H &\Rightarrow\ \Int \vNrv\ B\ \cap\ \Int\ \vNrv\ H\neq \emptyset.\mbox{\qquad \textcolor{blue}{\Squaresteel}}
\end{align*}
\end{example}

\begin{example}{\bf Spacetime Vortex Cycles: Overlapping Electromagnetic Vortices}.\\
I.V. Dzedolik observes that an electromagnetic vortex is formed by photons that possess some net angular momentum about the longitudinal axis of a dielectric waveguide~\cite[p. 135]{Dzedolik2005photonFluxVortexProperties}.   Photons are almost massless objects that carry energy from an emitter to an absorber~\cite{vanLeunen2018HilbertBookProject}.  Modeling spiraling vortices as vortex cycles equipped with the $\sconn$ proximity suggests the possibility of obtaining an expanded range of measurements in vortex optics.   N.M. Litchinitser observes that vortex-preshaped femtosecond laser pulses indicate the possibility of achieving repeatable and predictable spatial and temporal distribution in using metamaterials in light filamentation~\cite[p. 1055]{Litchinitser2012structuredLight}.   The overlap connectedness proximity space approach to characterizing, analysing and modelling neighboring photons gains strength by considering recent work by M. Hance on isolating and comparing different forms of photons (and photon vortical flux)~\cite[\S 4, pp. 8-11]{Hance2013photonPhysics}.
\qquad \textcolor{blue}{\Squaresteel}
\end{example}

\subsection{Descriptive Connectedness Proximity}
In this section, weak and strong descriptive connectedness proximities of skeletons arise when we consider pairs of vortex cycles with matching description.   A vortex cycle description is a feature vector that contains features values extracted from vortices with what are known as probe functions.   Let $K$ be a collection of vortex cycles equipped with the descriptive proximity $\dsconn$, which is an extension of the descriptive proximity $\snd$~\cite[\S 3-4, pp. 95-98]{DiConcilio2018MCSdescriptiveProximities}.   The mapping $\Phi:K\longrightarrow \mathbb{R}^n$ yields an $n$-dimensional feature vector in Euclidean space $\mathbb{R}^n$ either a vortex $\cyc A\in K$ (denoted by $\Phi(\cyc A)$) or a vortex cycle $\vcyc E$ in $K$ (denoted by $\Phi(\vcyc E)$) or a vortex nerve $\vNrv H$ in $K$ (denoted by $\Phi(\vNrv H)$).   For the axioms for a descriptive proximity, the usual set intersection is replaced by descriptive intersection~\cite[\S 3]{Peters2013Intersection} (denoted by $\dcap$) defined by
\[
 A \dcap B = \{x \in A \cup B: \Phi(x) \in \Phi(A), \ \Phi(x) \in \Phi(B) \}.
\]

The descriptive closure of $A$ (denoted by $\dcl A$)~\cite[\S 1.4, p. 16]{Peters2016ComputationalProximity} is defined by
\[
\dcl A = \left\{x\in K: x\dsconn A\right\}.
\]

The weak and strong forms of $\dsconn$ satisfy the following axioms.\\
\vspace{3mm}
\begin{description}
\item[{\bf P$_\Phi$4 [weak option]}]  $\Int A\ \dcap\ \Int B\neq \emptyset\ \Rightarrow\ A\ \dsconn\ B$.
\item[{\bf P$_\Phi$5 option]}] $A\ \dsconn\ B \Rightarrow A\dcap B\neq \emptyset$
\end{description}
\mbox{}\\
\vspace{3mm}

\begin{figure}[!ht]
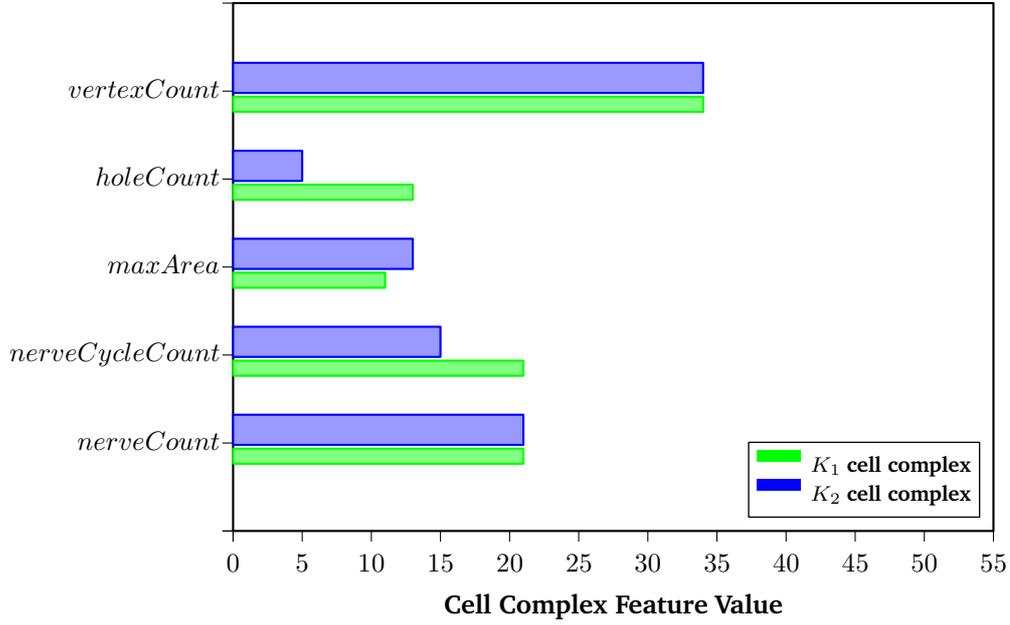

  \psset{llx=-1.5cm,lly=-1.5cm,xAxisLabel=\textbf{Cell Complex Feature Value}, xAxisLabelPos={c,-1cm},yAxisLabel=\textbf{},yAxisLabelPos={-1.5cm,c}, yLabels={,nerveCount, nerveCycleCount,maxArea,holeCount,vertexCount}}
\savedata{\mydata}[{21,0.85},{21,1.85},{11,2.85},{13,3.85},{34,4.85}]
\savedata{\data}[{21,1.15},{15,2.15},{13,3.15},{5,4.15},{34,5.15}]
\pslegend[rb]{\green\rule[1ex]{2em}{4pt} & $K_1\ \mbox{\bf cell complex}$\\ \blue\rule[1ex]{2em}{4pt} & $K_2\ \mbox{\bf cell complex}$\\}
\begin{psgraph}[axesstyle=frame,labels=x,ticksize=-4pt 0,Dx=5](0,0)(55,6){10cm }{7cm} \listplot[plotstyle=ybar,fillcolor=blue!40,linecolor=blue,barwidth=4mm, fillstyle=solid]{\data}

\listplot[plotstyle=ybar,fillcolor=green!50,linecolor=green,barwidth=2mm, fillstyle=solid,opacity=1]{\mydata} 
\end{psgraph}
\caption[]{Comparison of Cell Complex Feature Values}
\label{fig:55-vortexCycleFeatures}
\end{figure}
$\mbox{}$\\
\vspace{3mm}

Axiom {\bf P}$_\Phi${\bf 4} is a rewrite of the \u{C}ech axiom {\bf P4} and axiom {\bf P}$_\Phi${\bf 5} is an addition to the usual \u{C}ech axioms.
It is easy to see that $\dsconn$ satisfies the remaining \u{C}ech axioms after replacing $\delta$ with $\dsconn$.   Let $A,B,C\in K$, a cell complex space equipped with the proximity $\dsconn$, which satisfies the following axioms.\\
\vspace{3mm}

\noindent {\bf Descriptive Overlap Connectedness proximity axioms}.
\begin{compactenum}[{P$_\Phi$}1{dConn}]
\item $A\dcap B = \emptyset \ \Leftrightarrow A\ \fardsconn\ B$, {\em i.e.}, the sets of skeletons $A$ and $B$ are not descriptively close ($A$ and $B$ are far from each other).
\item $A\ \dsconn\ B\ \implies B\ \dsconn\ A$, {\em i.e.}, $A$ is descriptively close to $B$ implies $B$ is descriptively close to $A$.
\item $A\ \dsconn\ \left(B\cup C\right)\ \implies A\ \dsconn\ B\ \mbox{or}\ A\ \dsconn\ C$.
\item $\Int A\ \dcap\ \Int B\neq \emptyset\ \Rightarrow\ A\ \dsconn\ B$ (Weak Descriptive Connectedness Axiom).
\item  $A\ \dsconn\ B \Rightarrow A\dcap B\neq \emptyset$  (Strong Descriptive Connectedness Axiom). 
\qquad \textcolor{blue}{\Squaresteel}
\end{compactenum}
\mbox{}\\
\vspace{3mm}

A descriptive overlap connectedness space is denoted by $\left(K,\dsconn\right)$.   Skeletons $A,B$ in $K$ are close descriptively, provided the interior $\Int A$ has nonempty descriptive intersection with the interior $\Int A$.   This form of proximity has many applications, since we often want to compare objects such as 1-cycles by themselves or vortex cycles or the more complex vortex nerves that do not overlap spatially or at the same time.

\begin{figure}
\centering
\includegraphics[width=70mm]{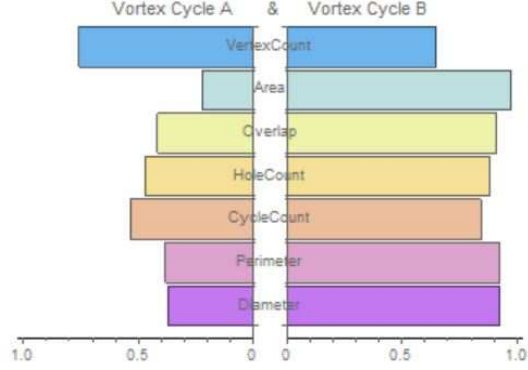}
\caption[]{Comparison of Vortex Cell Feature Values}
\label{fig:21-vortexCycleFeatures}
\end{figure}

\begin{example} {\bf Descriptive Connectedness Overlap of Disjoint Vortex Cycles in Spacetime}.\\
Let $\vcyc A, \vcyc B$ be a pair of vortex cycles in a collection of vortex cycles equipped with the proximities $\sconn$ and $\dsconn$.   A assume theses vortices represent non-overlapping electromagnetic vortexes that have matching descriptions in spacetime, {\em e.g.}, $\Phi(\vcyc A) = \Phi(\vcyc B) = $ (persistence duration).   That is, the length of time that $\vcyc A$ persists equals the duration of  $\vcyc B$.   In that case, $\vcyc A\ \dsconn\  \vcyc B$. 
 \qquad \textcolor{blue}{\Squaresteel}
\end{example}

\begin{example} {\bf Descriptive Connectedness Overlap of Cell Complexes}.\\
The bar graph\footnote{Many thanks to M.Z. Ahmad for the\ \LaTeX\ script used to display this bar graph, which does not depend on an external file.} in Fig.~\ref{fig:55-vortexCycleFeatures} compares feature values for a pair of cell complexes, namely,vertex count, hole count, maximum vortex cycle area, nerve cycle count and nerve count.   From the bar graph, $K_1\ \dsconn\ K_2$, since
\begin{align*}
\Phi(K_1\mbox{vertexCount}) &= \Phi(K_2\mbox{vertexCount}) = 35,\ \mbox{and}\\
\Phi(K_1\mbox{nerveCount}) &= \Phi(K_2\mbox{nerveCount}) = 21.
\end{align*}
This is the case, even though the hole count and nerve cycle count are far apart.
 \qquad \textcolor{blue}{\Squaresteel}
\end{example}

\begin{example} {\bf Absence of Descriptive Connectedness of Sample Vortex Cycles}.\\
The bar graph in Fig.~\ref{fig:21-vortexCycleFeatures} compares normalized feature values for a pair of sample vortex cycles $\vcyc A,\vcyc B$, namely, vertex count, vortex cycle area, overlap ({\em i.e.}, number of overlapping 1-cycles in a vortex cycle), hole count, cycle count, perimeter ({\em i.e.}, length of the boundary of a vortex cycle), diameter ({\em i.e.}, maximum distance between a pair of vertices on the boundary of a vortex cycle).   From the bar graph, it is apparent that $\vcyc A\ \fardsconn\ \vcyc B$, since there are no matching feature values for the sample pair of vortex cycles.
 \qquad \textcolor{blue}{\Squaresteel}
\end{example}

\begin{theorem}\label{thm:descriptiveOverlapVortexCyclesSpace}
Let $K$ be a collection of vortex cycles in a planar cell complex.   The space $K$ equipped with the relation $\dsconn$ is a proximity space.
\end{theorem}
\begin{proof}
The result follows from Lemma~\ref{lemma:skeletonProximitySpace}, since each vortex cycle in $K$ is also a collection of skeletons equipped with the proximity $\dsconn$.
\end{proof}

\begin{corollary}\label{thm:descriptiveOverlapVortexNervesSpace}
Let $K$ be a collection of vortex nerves in a planar cell complex.   The space $K$ equipped with the relation $\dsconn$ is a proximity space.
\end{corollary}
\begin{proof}
The result follows from Theorem~\ref{thm:descriptiveOverlapVortexCyclesSpace}, since each vortex nerve in $K$ is also a collection of intersecting vortex cycles equipped with the proximity $\dsconn$.
\end{proof}

\begin{example}\label{ex:matchingNerves}{\bf Non-Overlapping Vortex Nerve with Matching Descriptions}.\\
Let $K_{\vNrv}$ be a collection of vortex nerves in a planar cell complex the proximities $\conn$ and $\dsconn$.
Let $\vNrv A$ be a vortex nerve and let $\Phi(\vNrv A) = (\mbox{number of 1-cycles})$ be a description of the nerve based on one feature, namely, the number of 1-cycles in the nerve. Pairs of non-overlapping vortex nerves with matching descriptions are represented in Fig.~\ref{fig:7-overlappingVortexNerves}, namely, 
\begin{align*}
\vNrv A\ &\farconn\ \vNrv B\ \left(\mbox{Nerves $\vNrv A,\vNrv B$ do not overlap}\right),\\
\vNrv A\ &\dsconn\ \vNrv B,\ \mbox{since $\Phi\vNrv A) = \Phi(\vNrv B) = (2)$},\\
\vNrv A\ &\farconn\ \vNrv H\ \left(\mbox{Nerves  $\vNrv A,\vNrv H$ do not overlap}\right),\\
\vNrv A\ &\dsconn\ \vNrv H\ \mbox{since $\Phi(\vNrv A) = \Phi(\vNrv H) = (2)$},\\
\vNrv E\ &\farconn\ \vNrv B\ \left(\mbox{Nerves  $\vNrv E,\vNrv B$ do not overlap}\right),\\
\vNrv E\ &\dsconn\ \vNrv B\ \mbox{since $\Phi(\vNrv E) = \Phi(\cyc H_1) = (2)$},\\
\vNrv E\ &\farconn\ \vNrv H\ \left(\mbox{Nerves  $\vNrv E,\vNrv H$ do not overlap}\right),\\
\vNrv E\ &\dsconn\ \vNrv H\ \mbox{since $\Phi(\vNrv E) = \Phi(\vNrv H) = (2)$}.\mbox{\qquad \textcolor{blue}{\Squaresteel}} 
\end{align*}  
\end{example}

\begin{example}\label{ex:matchingCycles}{\bf Non-Overlapping Vortex Nerve Cycles with Matching Descriptions}.\\
Let $K_{\cyc}$ be a collection of 1-cycles in a planar cell complex the proximities $\conn$ and $\dsconn$.
Let $\cyc A$ be a 1-cycle in a vortex cycle and let $\Phi(\cyc A) = (\mbox{number of vertices})$ be a description of the cycle based on one feature, namely, the number of vertices in the cycle. Pairs of non-overlapping vortex nerves containing 1-cycles with matching descriptions are represented in Fig.~\ref{fig:7-overlappingVortexNerves}, namely, 
\begin{align*}
\cyc A_2\ &\farconn\ \cyc H_1\ \left(\mbox{Cycles $\cyc A_2,\cyc H_1$ do not overlap}\right),\\
\cyc A_2\ &\dsconn\ \cyc H_1,\ \mbox{since $\Phi(\cyc A_2) = \Phi(\cyc H_1) = (6)$},\\
\cyc A_2\ &\farconn\ \cyc B_2\ \left(\mbox{Cycles $\cyc A_2,\cyc B_2$ do not overlap}\right),\\
\cyc A_2\ &\dsconn\ \cyc B_2\ \mbox{since $\Phi(\cyc A_2) = \Phi(\cyc B_2) = (6)$},\\
\cyc A_1\ &\farconn\ \cyc H_1\ \left(\mbox{Cycles $\cyc A_1,\cyc H_1$ do not overlap}\right),\\
\cyc A_1\ &\dsconn\ \cyc H_1\ \mbox{since $\Phi(\cyc A_1) = \Phi(\cyc H_1) = (6)$},\\
\cyc A_1\ &\farconn\ \cyc B_2\ \left(\mbox{Cycles $\cyc A_1,\cyc B_2$ do not overlap}\right),\\
\cyc A_1\ &\dsconn\ \cyc B_2\ \mbox{since $\Phi(\cyc A_1) = \Phi(\cyc B_2) = (6)$}.\mbox{\qquad \textcolor{blue}{\Squaresteel}} 
\end{align*}  
\end{example}

\subsection{Vortex Cycle Spaces Equipped with Proximal Relators}
This section introduces a connectedness proximal relator~\cite{Peters2016relator} (denoted by $\mathscr{R}$), an extension of a Sz\'{a}z relator~\cite{Szaz1987}, which is a non-void collection of connectedness proximity relations on a nonempty cell complex $K$.   A space equipped with a proximal relator $\mathscr{R}$ is called a proximal relator space (denoted by $\left(K,\mathscr{R}\right)$).

\begin{example}{\bf Proximal Relator Space}.
Example~\ref{ex:matchingNerves} introduces a proximal relator space $\left(K_{\vNrv},\left\{\conn,\dsconn\right\}\right)$, useful in measuring, comparing, and classifying collections of vortex nerves that either have or do not have matching descriptions.   Similarly, Example~\ref{ex:matchingCycles} introduces a proximal relator  $\left(K_{\cyc},\left\{\conn,\dsconn\right\}\right)$, useful in the study of collections of 1-cycles that either have or do not have matching descriptions.
\qquad \textcolor{blue}{\Squaresteel}
\end{example}

The connection between $\sn$ and $\near$ is summarized in Lemma~\ref{thm:sconn-implies-near}.

\begin{lemma}\label{thm:sconn-implies-near}
Let $\left(K,\left\{\dsconn,\sconn,\conn\right\}\right)$ be a proximal relator space $K$, $A,B\subset K$.  Then 
\begin{compactenum}[{\rm (}$1${\rm )}]
\item $A\ \sconn\ B \Rightarrow A\ \conn\ B$.
\item $A\ \sconn\ B \Rightarrow A\ \dsconn\ B$.
\end{compactenum}
\end{lemma}
\begin{proof}$\mbox{}$\\
{\rm (}$1${\rm )}: From Axiom P5conn, $A\ \sconn\ B$ implies $A\ \cap\ B\neq \emptyset$, which implies $A\ \conn\ B$.   From Lemma~\ref{lemma:ConnectnessImpliesOverlap}, $A\ \conn\ B$ implies $A\ \cap\ B\neq \emptyset$, which implies $A\ \near\ B$ (from \u{C}ech Axiom P4).\\
{\rm (}$2${\rm )}: From {\rm (}$1${\rm )}, there are $\cyc\ x\in A, \cyc\ y\in B$\ common to $A$ and $B$.  Hence, $\Phi(\cyc\ x) = \Phi(\cyc\ y)$, which implies $A\ \dcap\ B\neq \emptyset$.  Then, from the descriptive connectedness Axiom $P_{\Phi}$4\mbox{conn}, $A\ \dcap\ B \neq \emptyset \Rightarrow\ A\ \dsconn\ B$. This gives the desired result. 
\end{proof}

Let $\vNrv A$ be a vortex nerve.   By definition, $\vNrv A$ is collection of 1-cycles with nonempty intersection.   The boundary of $\vNrv A$ (denoted by $\bdy \vNrv A$) is a sequence of connected vertices.   That is, for each pair of vertices $v,v'\in \bdy \vNrv A$, there is a sequence of edges, starting with vertex $v$ and ending with vertex $v'$.   There are no loops in  $\bdy \vNrv A$.   Consequently, $\bdy \vNrv A$ defines a simple, closed polygonal curve.   The interior of $\bdy \vNrv A$ is nonempty, since $\Nrv A$ is a collection of filled polytopes.   Hence, by definition, a $\vNrv A$ is also a nerve shape.

\begin{theorem}\label{thm:spoke}
Let $\left(K,\left\{\dsconn,\sconn\right\}\right)$ be a proximal relator space with nerve vortices $\vNrv A,\vNrv B\in K$.  Then
\begin{compactenum}[{\rm (}$1${\rm )}]
\item $\vNrv A\ \sconn\ \vNrv B$ implies $\vNrv A\ \dsconn\ \vNrv B$.
\item A 1-cycle $\cyc E\ \in\ \vNrv A\cap \vNrv B$ implies $\cyc E\ \in\ \vNrv A\ \dcap\ \vNrv B$.
\item A 1-cycle $\cyc E\ \in\ \vNrv A\cap \vNrv B$ implies $\vNrv A\ \dsconn\ \vNrv B$.
\end{compactenum}
\end{theorem}
\begin{proof}$\mbox{}$\\
{\rm (}$1${\rm )}: Immediate from part {\rm (}$2${\rm )} of Lemma~\ref{thm:sconn-implies-near}.\\
{\rm (}$2${\rm )}: By definition, $\vNrv A,\vNrv B$ are nerve shapes.   From Axioms {\bf P4conn, P5conn}, $\cyc E\ \in\ \vNrv A\cap \vNrv B$, if and only if $\vNrv A\ \dsconn\ \vNrv B$.   Consequently, $\cyc E$ is common to $\vNrv A,\vNrv B$.   Then there is a cycle $\cyc E\in \Nrv A$ with the same description as a cycle $\cyc E\in \vNrv B$.   Let $\Phi(\cyc E)$ be a description of $\cyc E$.   Then, $\Phi(\cyc E)\in \Phi(\vNrv A) \& \in \Phi(\cyc E)\in \Phi(\vNrv B)$, since $\cyc E\ \in\ \vNrv A\cap \vNrv B$.   Hence, $\cyc E\ \in\ \vNrv A\ \dcap\ \vNrv B$.\\
{\rm (}$3${\rm )}: Immediate from {\rm (}$2${\rm )} and Lemma~\ref{thm:sconn-implies-near}.
\end{proof}

\section{Main Results}
This section gives some main results for collections of proximal vortex cycles and proximal vortex nerves.

\subsection{Topology on Vortex Cycle Spaces}
This section introduces the construction of topology (homology) classes of vortex cycles and vortex nerves.  Topology classes have proved to be useful in classifying physical objects such as quasi-crystals~\cite{Dareau2017arXivCrystalTopClasses} and in knowledge extraction~\cite{Fermi2018knowledgeExtraction}. Such classes provide a basis for knowledge extraction about proximal vortex cycles and nerves.   A strong beneficial side-effect of the construction of such classes is the ease with which the persistence of homology class objects can be computed (see, {\em e.g.},~\cite{Fermi2017arXivPersistentTopology},~\cite{Taimanov2017LNAIoilGasPersistence}).   More importantly, the construction of topology classes leads to problem size reduction (see, {\em e.g.},~\cite[\S 3.1, p. 5]{Pellikka2010computationalHomologyElectromagnetics}).

\begin{lemma}\label{lemma:CWtopology}
Let $K$ be a nonempty collection of finite skeletons on a finite cell complex $K$ that is a Hausdorff space equipped the proximity $\conn$.  From the pair $\left(K,\conn\right)$, a Whitehead Closure Finite Weak (CW) Topology can be constructed. 
\end{lemma}
\begin{proof}$\mbox{}$\\
From Lemma~\ref{lemma:skeletonProximitySpace}, $\left(K,\conn\right)$ is a connectedness proximity space.  Let $\sk A, \sk B$ be skeletons in a finite cell complex $K$.  The closure $\cl(\sk A)$ is finite and includes the connected vertices on the boundary $\bdy(\sk A)$ and in the interior $\bdy(\sk A)$ of $\sk A$.   Since $K$ is finite, $\cl(\sk A)$ intersects a only a finite number of other skeletons in $K$.   The intersection $\sk A \cap \sk B\neq \emptyset$ is itself a finite skeleton, which can be either a single vertex or a set of edges common to $\sk A, \sk B$.  In that case, $\sk A\ \conn\ \sk B$.  By definition, $\sk A \cap \sk B$ is a skeleton in $K$.   Consequently, whenever $\sk A\ \conn\ \sk B$, then $\sk A \cap \sk B\in K$.   Hence, $\left(K,\conn\right)$ defines a Whitehead CW topology.
\end{proof}

\begin{theorem}\label{thm:vortexSconnCWtopology}
Let $K$ be a nonempty collection of finite skeletons on a finite cell complex $K$ that is a Hausdorff space equipped the proximity $\sconn$.  From the pair $\left(K,\sconn\right)$, a Whitehead Closure Finite Weak (CW) Topology can be constructed. 
\end{theorem}
\begin{proof}$\mbox{}$\\
Immediate from Lemma~\ref{lemma:CWtopology}.
\end{proof}


Next, we construct a Leader uniform topology on a collection of vortex cycles equipped with the descriptive connectedness proximity $\dsconn$.   

\begin{definition}
Let $X$ be a nonempty set.   For each given set $A\in 2^X$, form a cluster containing all subsets $B\in 2^X$ such that $A\cap B\neq \emptyset$.  The intersection as well as the union of clusters belong to $K$, defining a Leader uniform topology on $K$, namely, the collection of all uniform clusters on $K$.
\qquad \textcolor{blue}{\Squaresteel}
\end{definition}

\begin{theorem}\label{thm:vortexCyclesCWTopologySpatial}
Let $K$ be a finite collection of vortex cycles equipped the proximity $\sconn$ and let $\tau$ be a Leader uniform topology on the proximity space $\left(K,\sconn\right)$.   Then each cluster of vortex cycles $E\in \tau$ has a CW topology on $E$.
\end{theorem}
\begin{proof}$\mbox{}$\\
Each $E\in \tau$ is a finite collection of vortex cycles equipped with the proximity $\dsconn$.   Each closure $\cl(\vcyc H)\in E$ intersects with a finite number of other vortex cycles in $E$, since $E$ is finite (closure finiteness property).    Let $\cl(\vcyc A), \cl(\vcyc B)\in E$.   For $\Int(\vcyc A)\ \cap\ \Int(\vcyc B)\neq \emptyset \Rightarrow \cl(\vcyc A)\sconn\ \cl(\vcyc B)$, from Axiom P4intConn (weak topology property).   Hence, $E$ has a CW topology.
\end{proof}

For descriptive proximity spaces, the construction of Leader uniform topologies is accomplished by considering the descriptive intersection $\dcap$ and descriptive union $\dcup$ of nonempty sets of vortex cycles.  Let $K$ be a nonempty collection of vortex cycles, $A,B \in K$.   Then descriptive union $\dcup$ is defined by
\[
A\ \dcup B = \left\{E\in K: \Phi(E)\in \Phi\left(A\cup B\right) \right\}\ \mbox{({\bf Descriptive union of sets of vortex cycles})}.
\]

\begin{lemma}\label{thm:LeaderTopology}
Let $K$ be a nonempty collection of vortex cycles on a finite cell complex $K$ equipped the proximity $\dsconn$.  From the pair $\left(K,\dsconn\right)$, a Leader uniform topology can be constructed. 
\end{lemma}
\begin{proof}$\mbox{}$\\
We have $\Phi(K) = \left\{\Phi(\vcyc A): \vcyc A\in K\right\}$, the feature space for $K$.
Let $\vcyc A\ \dcap\ \vcyc B\neq \emptyset$ be descriptive intersection of a pair of vortex cycles  $\vcyc A, \vcyc B$ in $K$.   
From Axiom $P_{\Phi}4\mbox{conn}$, $\vcyc A\ \dsconn\ \vcyc B$.   
For each given $\vcyc A$, find all vortex cycles $\vcyc B\in K$ with nonempty intersection $\vcyc A\ \dcap\ \vcyc B\in \Phi(K)$ (intersection property), {\em i.e.}, all vortex cycles $\vcyc B$ such that $\vcyc A\ \dsconn\ \vcyc B$.  
 Let $A\ \dcup B = G$ be a descriptive union of sets of vortex cycles $A,B \in K$.   By definition, $\Phi(G)\in \Phi\left(A\cup B\right)$ (union property).   This gives the desired result.  
\end{proof}

\begin{theorem}\label{thm:vortexNervesLeaderTopology}
Let $K$ be a nonempty finite collection of vortex nerves equipped the proximity $\sconn$.  The proximity space $\left(K,\sconn\right)$ constructs a Leader uniform topology. 
\end{theorem}
\begin{proof}$\mbox{}$\\
Immediate from Lemma~\ref{thm:LeaderTopology}.
\end{proof}

From what we have observed so far, a form of problem reduction results from the construction of CW topology on a cluster in a Leader uniform topology.

\begin{theorem}\label{thm:clusterCWtopology}
Let $\mathcal{C}$ be a Leader uniform topology cluster in a collection of skeletons $K$  equipped the proximity $\conn$.  The proximity space $\left(\mathcal{C},\conn\right)$ constructs a CW topology. 
\end{theorem}
\begin{proof}$\mbox{}$\\
Immediate from Lemma~\ref{lemma:CWtopology}.
\end{proof}

\begin{corollary}
Let $\mathcal{C}$ be a Leader uniform topology cluster in a collection of skeletons $K$  equipped the proximity $\sconn$.  The proximity space $\left(\mathcal{C},\sconn\right)$ constructs a CW topology. 
\end{corollary}
\begin{proof}$\mbox{}$\\
Immediate from Theorem~\ref{thm:clusterCWtopology}.
\end{proof}

\begin{corollary}
Let $\mathcal{C}$ be a Leader uniform topology cluster in a collection of vortex cycles $K$  equipped the proximity $\sconn$.  The proximity space $\left(\mathcal{C},\sconn\right)$ constructs a CW topology. 
\end{corollary}
\begin{proof}$\mbox{}$\\
Immediate from Theorem~\ref{thm:clusterCWtopology}.
\end{proof}

\begin{corollary}
Let $\mathcal{C}$ be a Leader uniform topology cluster in a collection of vortex nerves $K$  equipped the proximity $\sconn$.  The proximity space $\left(\mathcal{C},\sconn\right)$ constructs a CW topology. 
\end{corollary}
\begin{proof}$\mbox{}$\\
Immediate from Theorem~\ref{thm:clusterCWtopology}.
\end{proof}

\subsection{Homotopic Types of Vortex Cycles and Vortex Nerves}


\begin{theorem}\label{EHnerve}{\rm ~\cite[\S III.2, p. 59]{Edelsbrunner1999}}
Let $\mathscr{F}$ be a finite collection of closed, convex sets in Euclidean space.  Then the nerve of $\mathscr{F}$ and the union of the sets in $\mathscr{F}$ have the same homotopy type.
\end{theorem}

\begin{lemma}\label{thm:vortexHomotopy}
Let $\cyc A$ be a vortex cycle in a finite collection of closed, convex skeletons in a cell complex $K$.
Then vortex cycle $\cyc A$ and the union of the skeletons in $\cyc A$ have the same homotopy type.
\end{lemma}
\begin{proof} 
From Theorem~\ref{EHnerve}, we have that the union of the skeletons $\sk E\in \cyc A$ and $\cyc A$ have the same homotopy type.
\end{proof}

\begin{theorem}\label{thm:vortexCyclesHomotopyType}
Let $K$ be a finite collection of vortex cycles equipped the proximity $\sconn$ and let $\tau$ be a Leader uniform topology on the proximity space $\left(K,\sconn\right)$.   Then each cluster of closed, convex vortex cycles $\mathcal{C}\in \tau$ and the union of vortex cycles in $\mathcal{C}$ have the same homotopy type.
\end{theorem}
\begin{proof} 
Each vortex cycle $\vcyc A$ in $\mathcal{C}$ is constructed from a collection of closed, convex skeletons in the cell complex $K$.   Consequently, $\mathcal{C}$ is a collection of closed, convex vortex cycles.   Hence, 
from Lemma~\ref{thm:vortexHomotopy}, we have that the union of the vortex cycles $\cyc A \in \mathcal{C}$ and $\mathcal{C}$ have the same homotopy type.
\end{proof}

\begin{corollary}\label{thm:vortexNervesHomotopyType}
Let $K$ be a finite collection of vortex nerves equipped the proximity $\sconn$ and let $\tau$ be a Leader uniform topology on the proximity space $\left(K,\sconn\right)$.   Then each cluster of closed, convex vortex nerves $\mathcal{N}\in \tau$ and the union of vortex nerves in $\mathcal{N}$ have the same homotopy type.
\end{corollary}
\begin{proof}$\mbox{}$\\
Immediate from Theorem~\ref{thm:vortexCyclesHomotopyType}, since vortex nerve is a collection of intersecting closed convex vortex cycles in $K$.
\end{proof}

\subsection{Open Problems}.
This section identifies open problems emerging from the study of proximal vortex cycles and proximal vortex nerves.   Vortex cycles can either be spatially close (overlapping vortex cycles have one or more common vertices) or descriptively close (pairs of vortex cycles that  intersect descriptively).   For such cell complexes, we have the following open problems.\\
\vspace{3mm}

\begin{compactenum}[{open}-1$^o$]
\item Vortex photons can be spatially close (overlap).    From Theorem~\ref{thm:vortexCyclesCWTopologySpatial}, a CW topology can be constructed on each cluster of vortex photons in a uniform Leader topology on a collection of vortex photons.   In that case, the problem of considering the spatial closeness of vortex photons for classification and analysis purposes, is simplified by considering a CW topology on each cluster of intersecting vortex photons.   This is a form of problem reduction, which has not yet been attempted.
\item The space between the spiraling flux of vortex photons can be viewed as holes.   Modelling vortex photons with holes using a combination of connectedness proximity and CW topology on clusters of such photons for classification and analysis purposes,  is an open problem.   This is a form of knowledge extraction.
\item It is well-known that real elementary particles can have the form of knots~\cite{FlamminiStasiak2007PMPESparticleKnots}, which have various forms in knot theory~\cite{Toiffoli2014knots}.   Vortex cycles can viewed as collections of intersecting knots.  The collection of all possible configurations of spatially close vortex cycles is an open problem.
\item A class of elementary particles known as glueballs exist as knotted chromodynamics flux lines~\cite{FlamminiStasiak2007PMPESparticleKnots}.   Vortex nerves can viewed as collections of intersecting (overlapping) glueballs.  The collection of all possible configurations of spatially close vortex nerves is an open problem.
\item From what has been observed in this paper, vortex cycles can be spatially close (overlap) vortex nerves.   The collection of all possible configurations of vortex cycles spatially close to vortex nerves is an open problem.
\item Let the cell complex $K$ be a Hausdorff space equipped with $\dsconn$ and descriptive closure $\dcl$.   Let $A$ be a cell (skeleton) in $K$.   A descriptive CW complex can be defined on each cell decomposition $A,B\in K$, if and only if
\begin{description}
\item [{\bf descriptive Closure Finiteness}] Closure of every cell (skeleton) $\dcl A$ intersects on a finite number of other cells.
\item [{\bf descriptive Weak topology}]  $A\in 2^K$ is descriptively closed ($A = \dcl A$), provided $A\dcap \dcl B$ is closed, {\em i.e.},
$
A\dcap \cl B = \dcl(A\cap \cl B).
$
\end{description}
Prove that $K$ has a topology $\tau$ that is a descriptive CW topology, provided $\tau$ has the descriptive closure finiteness and descriptive weak topology properties.
\item Let $K$ be a finite collection of vortex cycles that is a Hausdorff space equipped the proximity $\dsconn$ and descriptive closure $\dcl$ and let $\tau$ be a Leader uniform topology on the proximity space $\left(K,\dsconn\right)$.   Prove that each cluster of vortex cycles $E\in \tau$ has a descriptive CW topology on $E$.
\item Let $K$ be a finite collection of vortex nerves that is a Hausdorff space equipped the proximity $\dsconn$ and descriptive closure $\dcl$ and let $\tau$ be a Leader uniform topology on the proximity space $\left(K,\dsconn\right)$.   Prove that each cluster of vortex cycles $E\in \tau$ has a descriptive CW topology on $E$.
\item\label{prob:MNCcontours}  Inner and outer contours on maximal nucleus clusters (MNCs) on tessellated digital images~\cite[\S 8.9-8.2]{Peters2017computerVisionFoundations}form vortex cycles.   An open problem is to construct a CW topology on collections of MNC vortex cycles equipped with the relator $\left\{\conn,\sconn,\dsconn\right\}$.
\item An open problem to construct a Leader uniform topology on a collection of MNC vortex cycles equipped with the relator $\left\{\conn,\sconn,\dsconn\right\}$ and a CW topology on a Leader uniform topology cluster.
\item Brain tissue tessellation shows an absence of canonical microcircuits~\cite{PetersTozziRamanna2016NSLbrainTissueTessellation}.   For related work on donut-like trajectories along preferential brain railways, shaped as a torus, see, {\em e.g.},~\cite{TozziPeters2018NSLbrainRailways}.  An open problem is to construct a CW topology on a Leader uniform topology cluster (equipped with the proximity $\sconn$ or with $\dsconn$) that results from a brain tissue tessellation.   This is an application of the result from Problem~\ref{prob:MNCcontours}.
\item {\bf Vortex Cat in spacetime}.   By tessellating a video frame showing a cat, finding the maximum nucleus cluster MNC on the tessellated frame, and constructing fine and coarse contours surrounding the MNC nucleus, we obtain a vortex cycle.   By repeating these steps over a sequence of frames in a video, we obtain a vortex cat cycle in spacetime.   See, for example, the sample vortex cat cycles in~\cite{Enze2018vortexCatOne} and~\cite{Enze2018vortexCatTwo}.   An open problem is the construction of a Leader uniform topology on the collection of video frame vortex cat cycles equipped with the proximity $\sconn$ and to track the persistence of a Leader uniform topology cluster over a video frame sequence.
\item {\bf \u{C}ech nerve contours}.   Contours on \u{C}ech nerve nuclei are introduced in~\cite[\S 4.3.2, p. 119ff]{
AmadPeters2018TAMCSCechContours}.  An open problem is to construct a descriptive CW topology on a collection of  \u{C}ech nerve contours equipped with the proximity $\dsconn$.
\qquad \textcolor{blue}{\Squaresteel}
\end{compactenum}
$\mbox{}$\\
\vspace{3mm}
 
\bibliographystyle{amsplain}
\bibliography{NSrefs}

\end{document}